\documentclass[reqno,12pt]{amsart}

\usepackage{amsmath,amsthm,amssymb}
\newtheorem{theorem}{Theorem}[section]
\newtheorem{lemma}[theorem]{Lemma}
\newtheorem{corollary}[theorem]{Corollary}
\newtheorem{remark}[theorem]{Remark}
\newtheorem{example}[theorem]{Example}
\newtheorem{proposition}[theorem]{Proposition}
\theoremstyle{definition}
\newtheorem{definition}[theorem]{Definition}

\numberwithin{equation}{section}

\begin{document}

\title[A new approach to  Nikolskii--Besov classes]
{A new approach to  Nikolskii--Besov classes}

\author[Vladimir I. Bogachev et al.]{Vladimir~I.~Bogachev}
\address{}
\curraddr{}
\email{}
\thanks{
This research was supported by the Russian Science Foundation Grant 17-11-01058
at Lo\-mo\-no\-sov
Moscow State University.}

\author[]{Egor~D.~Kosov}
\address{}
\curraddr{}
\email{}
\thanks{The second author is a Young
Russian Mathematics award winner and would like to thank its sponsors and jury.}

\author[]{Svetlana N. Popova}
\address{}
\curraddr{}
\email{}
\thanks{}

\date{}

\maketitle

\begin{abstract}
We give a new characterization of   Nikolskii--Besov classes of functions of fractional smoothness  by means
of a nonlinear integration by parts formula in the form of a nonlinear inequality.
A similar characterization is obtained for
Nikolskii--Besov classes with respect to Gaussian measures on finite- and infinite-dimensional spaces.
\end{abstract}



\section{Introduction}

Nikolskii--Besov spaces play a very important role in the most
diverse aspects of analysis and applications
(see \cite{AdamsF}, \cite{BIN}, \cite{Leoni}, \cite{Nikol77}, \cite{Stein}, and~\cite{Triebel}).
Our recent papers
\cite{BKZd}, \cite{BKZ}, and \cite{Kos} on distributions of polynomials in Gaussian
random variables show that Nikolskii--Besov
spaces arise naturally also in the study of probability
distributions: it turns out that
the distribution densities of nonconstant polynomials in Gaussian random variables belong
to Nikolskii--Besov classes. Moreover, an extension of this result
to polynomials on spaces with general logarithmically concave measures
is obtained in~\cite{Kos}.
This crucial smoothness property
leads to many important consequences including
a sharper version of the known Nourdin--Poly estimate from~\cite{NourPol}.
See the survey  \cite{B16} for a detailed discussion.

In this paper, motivated by our  cited papers,  we suggest a new approach
to Nikolskii--Besov classes based on certain ``nonlinear integration by parts formulae''
in the form of nonlinear inequalities. In particular,
we establish a new characterization of
Nikolskii--Besov spaces by means of integration by parts such that
the classical description of the class $BV$ of functions of bounded variation
(see  \cite{AdamsF}, \cite{AFP}, and~\cite{Ziemer}) becomes a partial case.

We recall  (see \cite{BIN}, \cite{Nikol77}, and \cite{Stein})
that the Nikolskii--Besov space $B^\alpha_p(\mathbb{R}^n)$ with $p\in [1,+\infty)$
and $\alpha\in (0,1]$
consists of all functions $f\in L^p(\mathbb{R}^n)$
for which there is a constant $C$ such that for every $h\in \mathbb{R}^n$ one has
$$
\biggl(\int_{\mathbb{R}^d} |f(x+h)-f(x)|^p\, dx\biggr)^{1/p} \le  C|h|^\alpha,
$$
or
$$
\|f_h-f\|_p\le C|h|^\alpha,
$$
using the notation $f_h(x) := f(x-h)$.
The class $B^\alpha_p(\mathbb{R}^n)$ is a special case
of a general Besov class $B^\alpha_{p,\theta}(\mathbb{R}^n)$
with $\theta=\infty$.

The Nikolskii--Besov norm is defined by
$$
\|f\|_{B^\alpha_p(\mathbb{R}^n)} := \|f\|_p + \|f\|_{p, \alpha},
\quad \text{where } \|f\|_{p, \alpha}:=\sup_h |h|^{-\alpha}\|f_h-f\|_p.
$$
Note that for $\alpha=1$ we define   $B^1_p(\mathbb{R}^n)$ in the same way, which
differs from the classical Besov--Nikolskii space defined through
 the symmetrized difference $f(x+h)+f(x-h)-2f(x)$.
In particular, for $p=1$ this definition leads to the class~$BV$ of functions of bounded variation,
but not to the broader classical Nikolskii--Besov space.

Our first main result (presented in Section~\ref{sect2}) asserts that a function $f\in L^p(\mathbb{R})$
belongs to the class $B^\alpha_p(\mathbb{R})$ with
$0<\alpha\le 1$ precisely when there is a number $C$ such that
$$
\int \varphi'(x)f(x)\, dx \le C \|\varphi\|_q^\alpha \|\varphi'\|_q^{1-\alpha}
\quad \forall \, \varphi\in C_0^\infty(\mathbb{R}), \ \hbox{where } 1/p+1/q=1.
$$
In Section \ref{sect3} we extend the stated result to the multidimensional case in two different ways.
Firstly,
a function $f\in L^p(\mathbb{R}^n)$ belongs to the class $B^\alpha_p(\mathbb{R}^n)$ with
$0<\alpha\le 1$ if and only if
there is a number $C$ such that for every unit vector $e$ one has
$$
\int \partial_e\varphi(x)f(x)\, dx \le C \|\varphi\|_q^\alpha \|\partial_e\varphi\|_q^{1-\alpha}
\quad \forall \, \varphi\in C_0^\infty(\mathbb{R}^n).
$$
Secondly, another equivalent description is this:
there is a constant $C$ such that
$$
\int_{\mathbb{R}^n} \mathrm{div}\Phi(x) f(x)\, dx \le C\|\Phi\|^{\alpha}_q\|\mathrm{div}\Phi\|^{1-\alpha}_q
$$
for any vector field
$\Phi$ of class $C_0^\infty(\mathbb{R}^n, \mathbb{R}^n)$, where $q= p/(p-1)$.
This result opens a way to introducing Nikolskii--Besov classes on Riemannian
manifolds by means of similar nonlinear
integration by parts formulae with vector fields.
We note that the quantity
$$
V^{p,\alpha}(f):=
\sup \Bigl\{\|\Phi\|^{-\alpha}_q\|\mathrm{div}\Phi\|^{\alpha-1}_q
\int_{\mathbb{R}^n} \mathrm{div}\Phi(x) f(x)\, dx \Bigr\},
$$
where the supremum is taken over all vector fields $\Phi$ of class $C_0^\infty(\mathbb{R}^n, \mathbb{R}^n)$
with non-identically zero divergence,
is a fractional analog of the variation of a function of bounded variation.

In addition, as an application, in Section \ref{sect3} we  give a characterization of Nikolskii--Besov classes in terms
of the behavior of the heat semigroup
near zero, which  generalizes a classical result on functions of bounded variation
due to De Giorgi~\cite{DeG} and is close to the known characterizations
from \cite{Taib} and~\cite{Triebel}.

In Section \ref{sect4} we discuss the directional Nikolskii--Besov smoothness of functions on $\mathbb{R}^n$.
Here we construct an example
of a function $f$ on $\mathbb{R}^2$ such that it is Nikolskii--Besov smooth in the first variable,
but, for almost every fixed second variable~$y$, the
function $x\mapsto f(x,y)$ does not belong to the same Nikolskii--Besov class on the real line.
This surprising fact exhibits some difference between
Nikolskii--Besov classes  and Sobolev (or $BV$) classes, for which almost all restrictions retain
membership in the respective class.

Our approach, developed in Sections \ref{sect2} and \ref{sect3}, enables one to define
and study Nikolskii--Besov classes on Gaussian spaces
similarly to the approach of  \cite{FH}, \cite{AMMP1}, and \cite{AMMP2}
for the class of functions of bounded variation, where in place
of the classical divergence operator on $\mathbb{R}^n$ the Gaussian divergence operator is used.
This is done in Section~\ref{sect5} and Section~\ref{sect6}.
We deal with Nikolskii--Besov classes with respect to Gaussian measures in two steps.
Firstly, in Section~\ref{sect5}, we consider the finite-dimensional case, in which
all major technical and computational problems already occur.
Next, in Section~\ref{sect6}, we proceed to the infinite-dimensional case
by means of conditional expectations.
Our first main result concerning Gaussian Nikolskii--Besov classes, stated in Theorem~\ref{T3.2}, gives a
characterization via the Ornstein--Uhlenbeck semigroup,
which is a natural Gaussian  analog of the heat semigroup defined on $\mathbb{R}^n$ with Lebesgue measure.
We also obtain (in Theorem~\ref{T3.1}) a bound
for the rate of approximation by the Ornstein--Uhlenbeck semigroup
for functions from Nikolskii--Besov spaces.
From this estimate we derive a Poincar\'e-type inequality, where the fractional variation is used
in place of the norm of the gradient of a function. In addition, we obtain a
Gaussian analog of the fractional Hardy--Landau--Littlewood inequality (obtained in \cite{BKZ}) which generalizes
some results from
\cite{BWSH}, where such an inequality was obtained for the class of functions of bounded variation.

Finally, in Section~\ref{sect7} we introduce a natural analog of the Nikolskii--Besov smoothness
for measures on infinite-dimensional spaces.
The idea is  simple. The Skorohod
differentiability of a measure $\mu$ along a vector $h$ means that the mapping $t\mapsto \mu_{th}$
is Lipschitz with respect to the total variation distance,
where $\mu_{th}=\mu(\cdot-th)$ is the shifted measure.
If we now impose the $\alpha$-H\"older continuity in place of the Lipschitz condition,
we arrive at the definition of the Nikolskii--Besov $\alpha$-smoothness of a measure.
We prove that the space of all vectors of the Nikolskii--Besov $\alpha$-smoothness of a nonzero Radon measure
on a locally convex space is a complete metric vector space with respect to the distance
generated by the fractional directional variation and the embedding of this space into the
original space is compact.

A generalization of our new characterization to
the case of general Besov spaces, both classical and with respect to Gaussian measures,
will be considered in the forthcoming paper of the second author,
along with applications  to embedding theorems.

We thank O.V. Besov and B.S. Kashin for useful discussions.

Let us introduce some notation.
Throughout the paper
$C_0^\infty(\mathbb{R}^n)$ denotes the space
of all infinitely differentiable functions with compact support on $\mathbb{R}^n$ and
$C_b^\infty(\mathbb{R})$ denotes the space of all bounded infinitely
differentiable functions with bounded derivatives of every order.
For a function $f$ on $\mathbb{R}^n$ its $L^p$-norm is defined in the usual way
and in Sections \ref{sect2}, \ref{sect3}, and \ref{sect4} we will use the notation
$$
\|f\|_p := \biggl(\int_{\mathbb{R}^n} |f(x)|^p\, dx\biggr)^{1/p},\quad p\in[1,\infty),
\quad
\|f\|_\infty :=\sup_{x\in\mathbb{R}^n}|f(x)|.
$$
However, in Section \ref{sect5} and Section \ref{sect6} this notation will be
used by Gaussian measures in place of Lebesgue measure.

\section{Nikolskii--Besov classes on the real line}\label{sect2}

We start with the case of the real line to illustrate our approach to  Nikolskii--Besov classes
based on ``non-linear integration by parts''.

Throughout this section we assume that $1/p+1/q=1$, $p\in [1,\infty)$ and $\alpha$ is a fixed
number in the interval~$(0, 1]$. Infima over empty sets are by definition~$+\infty$.

\begin{definition}\label{D0.1}
Let $f\in L^p(\mathbb{R})$. Let $V^{p, \alpha}(f)$ be the infimum of all numbers $C$ such that
$$
\int_{\mathbb{R}} \varphi'(x) f(x)\, dx\le C\|\varphi\|_q^\alpha\|\varphi'\|_q^{1-\alpha}
\quad
\forall\, \varphi\in C_0^\infty(\mathbb{R}).
$$
\end{definition}

We observe that if the above estimate holds for all functions $\varphi\in C_0^\infty(\mathbb{R})$,
then by approximation it will be also valid for all functions $\varphi\in C^\infty(\mathbb{R})$
such that $\varphi\in L^q(\mathbb{R})$ and $\varphi'\in L^q(\mathbb{R})$.

The following theorem shows that the inclusion of a function $f$ to the class
$B^\alpha_p(\mathbb{R})$ is equivalent to the condition that the quantity $V^{p, \alpha}(f)$ is finite.
One implication from this theorem has been used in \cite{BKZ} and \cite{Kos}  to prove that, for any nonconstant
polynomial $f$ of degree $d$ on $\mathbb{R}^n$ and independent standard Gaussian random variables $\xi_1,\ldots,\xi_n$,
the distribution of the random variable $f(\xi_1,\ldots,\xi_n)$ has a density of class~$B^{1/d}_1$
and the order $1/d$ is sharp.

\begin{theorem}\label{T0.1}
Let $f\in L^p(\mathbb{R})$. Then
$f\in B^\alpha_p(\mathbb{R})$ if and only if
$V^{p,\alpha}(f)<\infty$.

Moreover,
$$
2^{\alpha -1}\|f\|_{p,\alpha} \le V^{p, \alpha}(f)\le ((1+\alpha)^{-1}+1)\|f\|_{p,\alpha}.
$$
\end{theorem}
\begin{proof}
Suppose that $V^{p,\alpha}(f)<\infty$.
Let $h>0$.
It is easy to see that
\begin{align*}
\|f_h-f\|_p &=\sup_{\substack{\varphi\in C_0^\infty\\ \|\varphi\|_q \le1}}
\int_{\mathbb{R}} \varphi(x)(f_h(x)-f(x))\, dx
=\sup_{\substack{\varphi\in C_0^\infty\\ \|\varphi\|_q \le1}}
\int_{\mathbb{R}} [\varphi(x+h)-\varphi(x)] f(x)\, dx
\\
&=
\sup_{\substack{\varphi\in C_0^\infty\\ \|\varphi\|_q \le1}}
\int_{\mathbb{R}} \int_0^{h}\varphi'(x+s) ds f(x) \, dx.
\end{align*}
Let us consider the following function of class $C_0^\infty(\mathbb{R})$:
$$
\psi(x)=\int_0^{h}\varphi(x+s)\, ds  .
$$
Note that $\|\psi\|_q \le |h|\|\varphi\|_q$
and that
$$
|\psi'(x)|=\biggl|\int_0^{h}\varphi'(x+se)\, ds\biggr|
=|\varphi(x+h)-\varphi(x)|,
$$
which yields the estimate $\|\psi'\|_q\le 2\|\varphi\|_q$.
By the assumptions of the theorem we have
$$
\int_{\mathbb{R}}\psi'(x)f(x)\, dx\le V^{p, \alpha}(f)\|\psi\|_q^\alpha\|\psi'\|_q^{1-\alpha}
\le 2^{1-\alpha}V^{p, \alpha}(f)|h|^\alpha\|\varphi\|_q.
$$
Therefore,
$$
\|f_h-f\|_p\le2^{1-\alpha}V^{p, \alpha}(f)|h|^\alpha,
$$
which completes the proof of one implication.

We now assume that $f\in B^\alpha_p(\mathbb{R})$.
For any function $\varphi\in C_0^\infty(\mathbb{R})$ we have
\begin{align*}
&\int_{\mathbb{R}}\varphi'(x)f(x)\, dx
=
\int_{\mathbb{R}}\varphi'(x)\bigl[f(x) - t^{-1}\int_0^tf(x-s)\, ds\bigr]\, dx
+
\int_{\mathbb{R}}\varphi'(x) t^{-1}\int_0^tf(x-s)\, ds\, dx
\\
&=t^{-1}\int_0^t\int_{\mathbb{R}}\varphi'(x)\bigl[f(x) -f(x-s)\bigr]\, dx \, ds
+
\int_{\mathbb{R}}t^{-1}\int_0^t\varphi'(x+s)\, ds\, f(x)\, dx
\\
&=
t^{-1}\int_0^t\int_{\mathbb{R}}\varphi'(x)\bigl[f(x) -f(x-s)\bigr]\, dx \, ds
+
\int_{\mathbb{R}}t^{-1}[\varphi(x+t) - \varphi(x)]f(x)\, dx
\\
&\le
\|\varphi'\|_q\|f\|_{p,\alpha} t^{-1}\int_0^t s^\alpha \, ds
+
t^{-1}\|\varphi\|_q\|f\|_{p,\alpha} t^\alpha
=
(1+\alpha)^{-1}\|\varphi'\|_q\|f\|_{p,\alpha} t^\alpha
+
\|\varphi\|_q\|f\|_{p,\alpha} t^{\alpha-1}.
\end{align*}
Now we take $t = \|\varphi\|_q\|\varphi'\|_q^{-1}$ and obtain the estimate
$$
\int_{\mathbb{R}}\varphi'(x)f(x)\, dx\le
((1+\alpha)^{-1}+1)\|f\|_{p,\alpha}\|\varphi\|_q^\alpha\|\varphi'\|_q^{1-\alpha},
$$
which completes the proof.
\end{proof}

\begin{remark}
{\rm
In case $\alpha=1$ and $p>1$,  the condition $V^{p,1}(f)<\infty$ is equivalent
to the inclusion of $f$ to the Sobolev class $W^{p,1}(\mathbb{R})$ (which consists
of all functions in $L^p(\mathbb{R})$ possessing locally absolutely continuous versions
with derivatives in $L^p(\mathbb{R})$).

In case $\alpha=1$  and $p=1$, the condition $V^{1,1}(f)<\infty$ is equivalent
to the inclusion of $f$ to the class $BV$ of integrable functions of bounded variation
(which consists of all integrable functions whose generalized derivatives are bounded measures).
}\end{remark}

\section{Nikolskii--Besov classes on $\mathbb{R}^n$}\label{sect3}

We proceed to  Nikolskii--Besov classes on $\mathbb{R}^n$ with Lebesgue measure.
Similarly to the one-dimensional case, we
obtain an equivalent description of Nikolskii--Besov spaces by means of ``nonlinear integration by parts''.
In addition,  we obtain a
characterization in terms of the heat semigroup.

We again assume that $1/p+1/q=1$, $p\in [1,\infty)$ and $\alpha$ is a fixed
number in the interval~$(0, 1]$.

Let
$\langle\cdot, \cdot\rangle$ be
the standard inner product on $\mathbb{R}^n$
and let $|\cdot|$ denote the corresponding norm.
Let $\{P_t\}_{t\ge0}$ denote the heat semigroup on $\mathbb{R}^n$ defined by
$$
P_tf(x) := (2\pi t)^{-n/2}\int_{\mathbb{R}^n}f(y)\exp\Bigl(-\frac {|x-y|^2}{2t}\Bigr)\, dy.
$$

We observe that Definition \ref{D0.1} admits two reasonable generalizations to the multidimensional case.
This is due to the fact that in the multidimensional case the function
 $\varphi'$ can be understood in two different ways:
as a partial derivative $\partial_e\varphi$ and as the divergence of a vector field ${\rm div}\Phi$.
This circumstance leads to the following two definitions.

\begin{definition}\label{D1.1}
Let $f\in L^p(\mathbb{R}^n)$. Let $V^{p, \alpha}(f)$ be the infimum of all numbers $C$ such that
$$
\int_{\mathbb{R}^n} \mathrm{div}\Phi(x) f(x) \, dx \le C\|\Phi\|^\alpha_q\|\mathrm{div}\Phi\|^{1-\alpha}_q
\quad
\forall\, \Phi\in C_0^\infty(\mathbb{R}^n, \mathbb{R}^n).
$$
\end{definition}

Clearly, the above estimate for vector fields of class $C_0^\infty(\mathbb{R}^n, \mathbb{R}^n)$
extends by approximation to all vector fields $\Phi=(\Phi_i)$ of class $C^\infty(\mathbb{R}^n, \mathbb{R}^n)$
such that $\Phi_i\in L^q(\mathbb{R}^n)$ and $\partial_{x_j}\Phi_i\in L^q(\mathbb{R}^n)$ for each~$j$ and~$i$.

\begin{definition}\label{D1.2}
Let $f\in L^p(\mathbb{R}^n)$. Let $V_0^{p, \alpha}(f)$ be the infimum of all numbers $C$ such that
$$
\int_{\mathbb{R}^n} \partial_e\varphi(x) f(x)\, dx\le C\|\varphi\|_q^\alpha\|\partial_e\varphi\|_q^{1-\alpha}
$$
for every function $\varphi\in C_0^\infty(\mathbb{R}^n)$
and every unit vector $e\in\mathbb{R}^n$.
\end{definition}

Obviously, this estimate extends to functions $\varphi\in C^\infty(\mathbb{R}^n)$
such that $\varphi\in L^q(\mathbb{R}^n)$ and $\partial_{x_j}\varphi\in~L^q(\mathbb{R}^n)$ for all~$j$.

The equivalence of the condition $V_0^{p, \alpha}(f)<\infty$ and the
inclusion $f\in B^\alpha_p(\mathbb{R}^n)$ can be verified similarly
to the proof of Theorem~\ref{T0.1}
(we provide details below).
The equivalence of these conditions to the condition $V^{p, \alpha}(f)<\infty$ is less obvious.
Here we need the following simple bound.

\begin{lemma}\label{lem1.1}
For any function $f\in B^\alpha_p(\mathbb{R}^n)$
one has
$$
\|f-P_tf\|_p\le c_{\alpha, n}\|f\|_{p,\alpha}t^{\alpha/2}, \quad
\hbox{where } c_{\alpha, n} = (2\pi)^{-n/2}\int_{\mathbb{R}^n} |z|^\alpha e^{-\frac{|z|^2}{2}}\, dz.
$$
\end{lemma}
\begin{proof}
We have
\begin{align*}
\|f - P_t f\|_p &= \sup\limits_{\substack{\varphi\in C_0^\infty\\ \|\varphi\|_q \le1}}
\int_{\mathbb{R}^n} \varphi(x)\biggl(f(x) - \int_{\mathbb{R}^n} f(x+\sqrt{t}z)(2\pi)^{-n/2}e^{-\frac{|z|^2}{2}}\,dz\biggr)\,dx
\\
&=\sup\limits_{\substack{\varphi\in C_0^\infty\\ \|\varphi\|_q \le1}}
\int_{\mathbb{R}^n}  (2\pi)^{-n/2}e^{-\frac{|z|^2}{2}}
\int_{\mathbb{R}^n} \varphi(x)\bigl(f(x)-f(x+\sqrt{t}z)\bigr)\,dx\,dz
\\
&\le
\|f\|_{p,\alpha}
t^{\alpha/2}(2\pi)^{-n/2}\int_{\mathbb{R}^n} |z|^\alpha e^{-\frac{|z|^2}{2}}\,dz,
\end{align*}
as announced.
\end{proof}

We now give an equivalent description of the Nikolskii--Besov spaces
in terms of nonlinear integration by parts inequalities.

\begin{theorem}\label{T1.1}
Let $f\in L^p(\mathbb{R}^n)$. The following conditions are equivalent{\rm:}

{\rm\rm(i)} $f\in B^\alpha_p(\mathbb{R}^n)${\rm;}

{\rm(ii)} $V^{p,\alpha}(f)<\infty${\rm;}

{\rm\rm(iii)} $V_0^{p,\alpha}(f)<\infty$.

Moreover,
$$
2^{\alpha-1}\|f\|_{p,\alpha}\le V_0^{p, \alpha}(f)\le V^{p, \alpha}(f)\le C(n, \alpha)\|f\|_{p,\alpha},
$$
where
$$
C(n,\alpha) = (2\pi)^{-n/2}\int_{\mathbb{R}^n}
 |z|^\alpha e^{-\frac{|z|^2}{2}}dz + (2\pi)^{-n/2}\int_{\mathbb{R}^n} |z|^{1+\alpha}e^{-\frac{|z|^2}{2}}\,
 dz \le \sqrt{n} + n.
$$
\end{theorem}
\begin{proof}
As we have already noted,
the equivalence $\rm(i)\Leftrightarrow\rm(iii)$ can be proved similarly to the one-dimensional case.
We provide details for completeness.

$\rm(ii)\Rightarrow(iii)$.
For every function $\varphi\in C_0^\infty(\mathbb{R}^n)$ and for every unit vector $e\in\mathbb{R}^n$
we  take the vector field $\Phi=e\varphi$ and conclude that
$$
V_0^{p, \alpha}(f) \le V^{p, \alpha}(f).
$$

$\rm(iii)\Rightarrow(i)$.
Let $e = |h|^{-1}h$.
It is easy to see that
\begin{align*}
\|f_h-f\|_p &=\sup_{\substack{\varphi\in C_0^\infty\\ \|\varphi\|_q \le1}}
\int_{\mathbb{R}^n} \varphi(x)(f_h(x)-f(x))\, dx
=\sup_{\substack{\varphi\in C_0^\infty\\ \|\varphi\|_q \le1}}
\int_{\mathbb{R}^n} [\varphi(x+h)-\varphi(x)] f(x)\, dx
\\
&=
\sup_{\substack{\varphi\in C_0^\infty\\ \|\varphi\|_q \le1}}
\int_{\mathbb{R}^n} \biggl(\int_0^{|h|}\partial_e\varphi(x+se)\, ds\biggr) f(x)\, dx.
\end{align*}
As above, we consider the following function of class $C_0^\infty(\mathbb{R}^n)$:
$$
\psi(x)=\int_0^{|h|}\varphi(x+se)\, ds  .
$$
Note that $\|\psi\|_q \le |h|\|\varphi\|_q$
and that
$$
|\partial_e\psi(x)|=\biggl|\int_0^{|h|}\partial_e\varphi(x+se)\,ds\biggr|
=|\varphi(x+h)-\varphi(x)|,
$$
which provides the bound $\|\partial_e\psi\|_q\le 2\|\varphi\|_q$.
By Definition \ref{D1.2} we have
$$
\int_{\mathbb{R}^n}\partial_e\psi(x)f(x)\, dx
\le V_0^{p, \alpha}(f)\|\psi\|_q^\alpha\|\partial_e\psi\|_q^{1-\alpha}
\le 2^{1-\alpha}V_0^{p, \alpha}(f)|h|^\alpha\|\varphi\|_q,
$$
which implies the estimate
$$
\|f_h-f\|_p\le2^{1-\alpha}V_0^{p, \alpha}(f)|h|^\alpha.
$$

$\rm(i)\Rightarrow(ii)$.
For every
vector field $\Phi\in C_0^\infty(\mathbb{R}^n, \mathbb{R}^n)$
we have
$$
\int_{\mathbb{R}^n} \mathrm{div}\Phi(x) f(x)\, dx =
\int_{\mathbb{R}^n} \mathrm{div}\Phi(x) (f(x) - P_tf(x))\, dx
 + \int_{\mathbb{R}^n} \mathrm{div}\Phi(x) P_tf(x)\, dx.
$$
For the first term on the right by Lemma \ref{lem1.1} we have
$$
\int_{\mathbb{R}^n} \mathrm{div}\Phi(x) (f(x) - P_tf(x))\, dx \le \|\mathrm{div}\Phi\|_q\|f - P_tf\|_p\le
c_{\alpha, n}\|f\|_{p,\alpha}\|\mathrm{div}\Phi\|_qt^{\alpha/2}.
$$
For the second term we have
\begin{align*}
\int_{\mathbb{R}^n} \mathrm{div}\Phi(x) P_tf(x)\, dx &=
- \int_{\mathbb{R}^n}  \langle \Phi(x), \nabla P_tf(x)\rangle\, dx
\\
&= t^{-1/2}\int_{\mathbb{R}^n}  \int_{\mathbb{R}^n}
\langle \Phi(x), (x-y)t^{-1/2}\rangle f(y)(2\pi t)^{-n/2}e^{-\frac{|x-y|^2}{2t}}\, dy\, dx
\\
&= t^{-1/2}\int_{\mathbb{R}^n}  \int_{\mathbb{R}^n}
\langle \Phi(x), z\rangle f(x - \sqrt{t}z)(2\pi)^{-n/2}e^{-\frac{|z|^2}{2}}\, dz\, dx
\\
&=
t^{-1/2}\int_{\mathbb{R}^n}  (2\pi)^{-n/2}e^{-\frac{|z|^2}{2}} \int_{\mathbb{R}^n}
\langle \Phi(x), z\rangle (f(x - \sqrt{t} z) - f(x))\, dx\, dz
\\
&\le
t^{-1/2}t^{\alpha/2}\|f\|_{p,\alpha}\int_{\mathbb{R}^n}  (2\pi)^{-n/2}e^{-\frac{|z|^2}{2}}
\biggl(\int_{\mathbb{R}^n}  \langle \Phi(x), z\rangle^q dx\biggr)^{1/q}|z|^\alpha \, dz
\\
&\le t^{(-1+\alpha)/2}\|f\|_{p,\alpha}\|\Phi\|_q(2\pi)^{-n/2}
\int_{\mathbb{R}^n} |z|^{1+\alpha}e^{-\frac{|z|^2}{2}}\, dz.
\end{align*}
In the fourth equality above we have used the fact that the
double integral of $f(x)\langle \Phi(x), z\rangle e^{-\frac{|z|^2}{2}}$ in $z$ and $x$ vanishes,
since the integral in $z$
vanishes for every fixed~$x$.
Thus, we have
$$
\int_{\mathbb{R}^n} \mathrm{div}\Phi(x) f(x)\, dx\le
t^{\alpha/2}\|\mathrm{div}\Phi\|_q\|f\|_{p,\alpha}c_{\alpha, n} +
t^{(-1+\alpha)/2}\|\Phi\|_q\|f\|_{p,\alpha}(2\pi)^{-n/2}\int|z|^{1+\alpha}e^{-\frac{|z|^2}{2}}\, dz.
$$
Taking $t = \|\Phi\|^2_q\|\mathrm{div}\Phi\|_q^{-2}$ we find that
$$
\int_{\mathbb{R}^n} \mathrm{div}\Phi(x) f(x) \, dx
\le C(n, \alpha)\|f\|_{p,\alpha}\|\Phi\|^\alpha_q\|\mathrm{div}\Phi\|^{1-\alpha}_q
$$
with $C(n,\alpha)$ given by the announced expression.
 \end{proof}

 \begin{remark}
{\rm
Let $\alpha=1$. If $p>1$, then $B^1_p(\mathbb{R}^n)$ coincides with
 the Sobolev class $W^{p,1}(\mathbb{R}^n)$ consisting of functions
belonging to $L^p(\mathbb{R}^n)$ along with their generalized first order partial derivatives.
But if $p=1$, then $B^1_p(\mathbb{R}^n)$ coincides with
the class $BV(\mathbb{R}^n)$ of functions of bounded variation consisting of integrable
functions whose generalized first order partial derivatives are bounded measures.
}\end{remark}

Let us proceed to a characterization of the Nikolskii--Besov classes in terms of the heat semigroup.
The proof employs the previous theorem.

\begin{definition}\label{D1.3}
For a function $f\in L^p(\mathbb{R}^n)$ set
$$
U^{p, \alpha}(f) := \sup_{t>0} t^{\frac{1-\alpha}{2}}\|\nabla P_tf\|_p,
$$
where $\{P_t\}_{t\ge 0}$ is the heat semigroup.
\end{definition}

We need the following analog of Lemma \ref{lem1.1}.

\begin{lemma}\label{lem1.2}
Let $f\in L^p(\mathbb{R}^n)$ be such that  $U^{p, \alpha}(f)<\infty$.
Then
$$
\|f-P_tf\|_p\le 4\alpha^{-1}\sqrt{n}U^{p,\alpha}(f)t^{\alpha/2} \quad \forall\, t\ge0.
$$
\end{lemma}
\begin{proof}
For any function $\varphi\in C_0^\infty(\mathbb{R}^n)$ we have
\begin{multline*}
\int_{\mathbb{R}^n} \varphi(x)(P_tf(x) -f(x))\, dx
= \int_{\mathbb{R}^n} (P_t\varphi(x) -\varphi(x))f(x)\, dx
= \int_{\mathbb{R}^n} \int_0^t \frac{\partial}{\partial s}P_s\varphi(x) \, ds f(x)\, dx
\\
= \int_0^t \int_{\mathbb{R}^n} {\rm div}\nabla P_s\varphi(x) f(x)\, dx\, ds
= \int_0^t \int_{\mathbb{R}^n} P_{s/2}{\rm div}\nabla P_{s/2}\varphi(x) f(x)\, dx\, ds
\\
= \int_0^t \int_{\mathbb{R}^n} {\rm div}\nabla P_{s/2}\varphi(x) P_{s/2}f(x)\, dx\, ds
= -\int_0^t \int_{\mathbb{R}^n} \langle\nabla P_{s/2}\varphi(x), \nabla P_{s/2}f(x)\rangle\, dx\, ds.
\end{multline*}
The gradient $\nabla P_{s/2}\varphi$ admits the representation
\begin{align*}
\nabla P_{s/2}\varphi(x) &= -(\pi s)^{-n/2}
\int_{\mathbb{R}^n} \varphi(y) (x - y)2s^{-1}\exp\biggl(-\frac{|x-y|^2}{s}\biggr)\, dy
\\
&=
(s/2)^{-1/2}(2\pi)^{-n/2}\int_{\mathbb{R}^n} \varphi(x+\sqrt{s/2}z) ze^{-\frac{|z|^2}{2}}\, dz.
\end{align*}
Therefore,
$$
\|\nabla P_{s/2}\varphi\|_q\le(s/2)^{-1/2}\|\varphi\|_q(2\pi)^{-n/2}
\int_{\mathbb{R}^n} |z|e^{-\frac{|z|^2}{2}}\,dz
\le\sqrt{n}(s/2)^{-1/2}\|\varphi\|_q.
$$
Thus, we have
\begin{align*}
\int_{\mathbb{R}^n} \varphi(x)(f(x) - P_tf(x))\, dx  &\le
\int_0^t \|\nabla P_{s/2}\varphi\|_q \|\nabla P_{s/2}f\|_p\, ds
\\
&\le
\sqrt{n}\|\varphi\|_q U^{p, \alpha}(f) \int_0^t(s/2)^{\frac{-1+\alpha}{2}}(s/2)^{-1/2}\, ds
\\
&=
\sqrt{n}\|\varphi\|_q U^{p, \alpha}(f) \int_0^t(s/2)^{\frac{\alpha}{2}-1}\, ds
\le
4\alpha^{-1}\sqrt{n}\|\varphi\|_q U^{p, \alpha}(f) t^{\alpha/2}.
\end{align*}
Now taking the supremum over functions $\varphi$ with $\|\varphi\|_q\le 1$ we obtain the announced estimate.
\end{proof}

We now prove the following result, which is a characterization of the Nikolskii--Besov classes
in terms of the behavior of the heat semigroup near zero.
This result is close to the known characterization from \cite{Taib} and~\cite{Triebel}
that employs certain norms of $\partial_t T_t f$ rather than the gradient of~$T_tf$.

\begin{theorem}\label{T1.2}
Let $f\in L^p(\mathbb{R}^n)$. Then
$V^{p,\alpha}(f)<\infty$ if and only if
$U^{p, \alpha}(f)<\infty$.

Moreover,
$$
U^{p, \alpha}(f) \le n^{\frac{1-\alpha}{2}}V^{p, \alpha}(f),
\quad
V^{p, \alpha}(f)\le (4\sqrt{n}\alpha^{-1} + 1) U^{p, \alpha}(f).
$$
\end{theorem}
\begin{proof}
Assume that $V^{p,\alpha}(f)<\infty$.
For every
vector field $\Phi=(\Phi_i)\in C_0^\infty(\mathbb{R}^n, \mathbb{R}^n)$
with $\|\Phi\|_q\le 1$ one can easily verify that
$$
\int_{\mathbb{R}^n} \langle \Phi(x), \nabla P_tf(x)\rangle\, dx =
-\int_{\mathbb{R}^n} {\rm div}\bigl(P_t\Phi\bigr)(x)f(x)\, dx.
$$
The right-hand side is dominated by
$$
V^{p, \alpha}(f)\|P_t\Phi\|^\alpha_q\|\mathrm{div}P_t\Phi\|^{1-\alpha}_q,
$$
since $P_t\Phi_i\in C^\infty(\mathbb{R}^n)$, $P_t\Phi_i\in L^q(\mathbb{R}^n)$ and
$\partial_{x_j} P_t\Phi_i\in L^q(\mathbb{R}^n)$ for all~$i,j$.
Let us recall that $\|P_t\Phi\|_q\le \|\Phi\|_q\le 1$. Let us estimate the norm of the divergence.
We have
\begin{align*}
\frac{\partial}{\partial x_j}P_t\Phi_j(x) &=
-(2\pi t)^{-n/2}\int_{\mathbb{R}^n} \Phi_j(y) (x_j - y_j)t^{-1}\exp\biggl(-\frac{|x-y|^2}{2t}\biggr)\, dy
\\
&=
t^{-1/2}(2\pi)^{-n/2}\int_{\mathbb{R}^n} \Phi_j(x+\sqrt{t}z) z_je^{-\frac{|z|^2}{2}}\, dz.
\end{align*}
Thus, by  Minkowski's inequality for integrals (see, e.g., \cite[V.~1, p.~231]{mera})
\begin{align*}
\|\mathrm{div}P_t\Phi\|_q
&= t^{-1/2}\biggl(\int_{\mathbb{R}^n}
\biggl|(2\pi)^{-n/2}\int_{\mathbb{R}^n}\langle\Phi(x+\sqrt{t}z), z\rangle
e^{-\frac{|z|^2}{2}}\, dz\biggl|^q \, dx\biggr)^{1/q}
\\
&\le
t^{-1/2}\biggl(\int_{\mathbb{R}^n}\biggl(
(2\pi)^{-n/2}\int_{\mathbb{R}^n}|\Phi(x+\sqrt{t}z)|\,|z|
e^{-\frac{|z|^2}{2}}\, dz\biggr)^q\, dx\biggr)^{1/q}
\\
&\le
t^{-1/2}(2\pi)^{-n/2}\int_{\mathbb{R}^n}
\biggl(\int_{\mathbb{R}^n}|\Phi(x+\sqrt{t}z)|^q \, dx\biggr)^{1/q} |z|e^{-\frac{|z|^2}{2}}\, dz
\\
&=
t^{-1/2}\|\Phi\|_q (2\pi)^{-n/2}\int_{\mathbb{R}^n}|z| e^{-\frac{|z|^2}{2}}\, dz
\le \sqrt{n}t^{-1/2}\|\Phi\|_q.
\end{align*}
Summing up these estimates we obtain the bound
$$
\int_{\mathbb{R}^n} \langle \Phi(x), \nabla P_tf(x)\rangle \, dx\le
t^{-\frac{1-\alpha}{2}}n^{\frac{1-\alpha}{2}}V^{p, \alpha}(f).
$$
Taking the supremum over vector fields $\Phi$ we obtain the desired estimate.

Now we assume that $U^{p, \alpha}(f)<\infty$.
For every
vector field $\Phi\in C_0^\infty(\mathbb{R}^n, \mathbb{R}^n)$
we have
$$
\int_{\mathbb{R}^n} {\rm div}\Phi(x) f(x)\, dx
= \int_{\mathbb{R}^n} {\rm div}\Phi(x) (f(x) - P_tf(x))\, dx
+ \int_{\mathbb{R}^n} {\rm div}\Phi(x) P_tf(x)\, dx.
$$
Using Lemma \ref{lem1.2} we estimate the first term on the right as follows:
$$
\int_{\mathbb{R}^n} {\rm div}\Phi(x) (f(x) - P_tf(x))\, dx\le\|{\rm div}\Phi\|_q\|f - P_tf\|_p\le
4\sqrt{n}\alpha^{-1}U^{p,\alpha}(f)t^{\alpha/2}\|{\rm div}\Phi\|_q.
$$
For the second term we have
$$
\int_{\mathbb{R}^n} {\rm div}\Phi(x) P_tf(x)\, dx =
- \int_{\mathbb{R}^n} \langle\Phi(x), \nabla P_tf(x)\rangle \, dx\le
\|\Phi\|_q\|\nabla P_tf\|_p\le U^{p,\alpha}(f)t^{(\alpha-1)/2}\|\Phi\|_q.
$$
Therefore,
$$
\int_{\mathbb{R}^n} {\rm div}\Phi(x) f(x)\, dx
 \le 4\sqrt{n}\alpha^{-1}U^{p,\alpha}(f)t^{\alpha/2}\|{\rm div}\Phi\|_q
+ U^{p,\alpha}(f)t^{(\alpha-1)/2}\|\Phi\|_q.
$$
Taking $t = \|\Phi\|_q^2\|{\rm div}\Phi\|_q^{-2}$
we obtain the announced bound.
\end{proof}

Below we consider yet another classical semigroup: the Ornstein--Uhlenbeck semigroup
in the Gaussian case.

\section{Directional Nikolskii--Besov smoothness}\label{sect4}

In this section we discuss the directional Nikolskii--Besov smoothness of functions, in particular,
the fractional smoothness with respect to a given  variable (but in the last section
we shall see that this property is naturally defined for measures that need not be absolutely continuous).
 Example~\ref{ex1} below shows that the directional Nikolskii--Besov smoothness
of a function on the plane with respect to the first variable does not imply inclusion
to the respective Nikolskii--Besov class
of the restrictions of this function to the straight lines parallel to the first coordinate line.
This is in contrast to the behavior
of  Sobolev or $BV$ functions: we recall that if a function $f$ belongs to the Sobolev
class $W^{p,1}(\mathbb{R}^n)$ or to the class $BV(\mathbb{R}^n)$, then, for almost
all fixed values $x_1,\ldots,x_{n-1}$, the function $x_n\mapsto f(x_1,\ldots,x_n)$ belongs to
$W^{p,1}(\mathbb{R})$ or $BV(\mathbb{R})$, respectively (see~\cite{DM}).

It is quite natural to define the
directional Nikolskii--Besov smoothness of a function on $\mathbb{R}^n$ in the following way.

\begin{definition}\label{D1.4}
Let $\alpha \in (0,1]$, $p\in [1,+\infty)$, $e\in\mathbb{R}^n$, and $f\in L^p(\mathbb{R}^n)$. Set
$$
\|f\|_{p, \alpha; e} = \sup_{t\in\mathbb{R}}|t|^{-\alpha}\|f_{te} - f\|_p.
$$
We say that the function $f$ is Nikolskii--Besov $\alpha$-smooth along $e$
if the quantity $\|f\|_{p, \alpha; e}$ is finite.
\end{definition}

It is straightforward to introduce a directional analog of Definition \ref{D1.2}.

\begin{definition}\label{D1.5}
Let $f\in L^p(\mathbb{R}^n)$. Let $V^{p, \alpha}(f; e)$ be the infimum of numbers $C$ such that
$$
\int_{\mathbb{R}^n} \partial_e\varphi(x) f(x) dx\le C\|\varphi\|_q^\alpha\|\partial_e\varphi\|_q^{1-\alpha}
\quad
\forall\, \varphi\in C_0^\infty(\mathbb{R}^n).
$$
\end{definition}

For the directional Nikolskii--Besov smoothness we have the following analog of Theorem \ref{T0.1}
proved by the same reasoning.

\begin{theorem}\label{T1,5.1}
Let $f\in L^p(\mathbb{R}^n)$. Then
$\|f\|_{p,\alpha; e}<\infty$ if and only if
$V^{p,\alpha}(f; e)<\infty$.
\end{theorem}

The following example shows that there is a function $f$ on $\mathbb{R}^2$ that is
Nikolskii--Besov $\alpha$-smooth along the vector $e_1=(1,0)$, but the functions $x \mapsto f(x, y)$
are not Nikolskii--Besov $\alpha$-smooth for almost every~$y$.
While Theorem \ref{T1.1} shows some similarity between Nikolskii--Besov classes
and the class $BV$, this example exhibits a difference between them with respect
to restrictions. However, in the last section we shall see that the fractional smoothness is inherited
by restrictions with an arbitrarily small loss of the order of smoothness.

\begin{example}\label{ex1}
{\rm
Fix $\alpha\in(0,1)$.
Let $\{J_k\}$ be intervals in $[0,1]$ such that $|J_k| = k^{-1}(\ln k)^{-1}$
and each point in $[0,1]$ is covered by infinitely many of such intervals
(this can be done since $\sum_k |J_k| = \infty$).
Set $f_k(y) = k^{-\alpha}\sqrt{\ln k} I_{k}(y)$, where $I_k$ is the indicator function of~$J_k$.
We note that $\sum_{k=1}^\infty \|f_k\|^2_{L^2(\mathbb{R})} < \infty$.
Consider a function $f$ on $\mathbb{R}^2$ defined by
$$
f(x, y) = \sum_{k=1}^\infty I_{[0,2\pi]}(x)\sin (kx)f_k(y).
$$
The function $f$ is well-defined as an element of $L^2(\mathbb{R}^2)$,
since the series converges in $L^2(\mathbb{R}^2)$.
For every $k$, for almost every $y\in[0,1]$ by Fubini's theorem we have
$$
\int_{\mathbb{R}} \sin(kx) f(x, y)\, dx = \pi f_k(y).
$$
It follows that for almost every $y\in [0,1]$, the function $x\mapsto f(x, y)$ does not belong
to $B^\alpha_1(\mathbb{R})$. Indeed, if this function belongs to this space, then
by Theorem \ref{T0.1} we have the estimate
$$
f_k(y) = \pi^{-1}\int_{\mathbb{R}} \sin(kx) f(x, y)\, dx \le C k^{-\alpha},
$$
but in our case for infinitely many $k$ one has $f_k(y) = k^{-\alpha}\sqrt{\ln k}$, which leads to a contradiction.

We now observe that for any function $\varphi\in C^\infty_0(\mathbb{R}^2)$ one has
\begin{align*}
&\int_{\mathbb{R}}\int_{\mathbb{R}} \partial_x\varphi(x,y) f(x,y)\, dx\, dy=
\sum_{k=1}^\infty\int_0^1  f_k(y)\int_0^{2\pi} \partial_x\varphi(x,y) \sin (kx)\, dx\, dy
\\
&=
\int_0^1\int_0^{2\pi}  \varphi(x,y)\Bigl[-\sum_{k=1}^Nf_k(y) k\cos (kx)\Bigr]\, dx\, dy
+
\int_0^1\int_0^{2\pi} \partial_x\varphi(x,y)\Bigl[\sum_{k=N+1}^\infty f_k(y) \sin (kx)\Bigr]\, dx\, dy
\\
&\le
\sqrt{2\pi}\|\varphi\|_\infty\Bigl\|\sum_{k=1}^Nf_k(y) k\cos (kx)\Bigr\|_{L^2([0,1]\times[0,2\pi])}
+
\sqrt{2\pi}\|\partial_x\varphi\|_\infty\Bigl\|\sum_{k=N+1}^\infty f_k(y) \sin (kx)\Bigr\|_{L^2([0,1]\times[0,2\pi])}
\\
&=
\sqrt{2}\pi\|\varphi\|_\infty\Bigl(\sum_{k=1}^N k^{-2\alpha-1} k^2\Bigr)^{1/2}
+
\sqrt{2}\pi\|\partial_x\varphi\|_\infty\Bigl(\sum_{k=N+1}^\infty k^{-2\alpha-1} \Bigr)^{1/2}
\\
&\le
C(\|\varphi\|_\infty N^{1-\alpha} + \|\partial_x\varphi\|_\infty N^{-\alpha}).
\end{align*}
Taking $N=\bigl[\|\partial_x\varphi\|_\infty\|\varphi\|_\infty^{-1}\bigr]$ we obtain
$$
\int_{\mathbb{R}}\int_{\mathbb{R}} \partial_x\varphi(x,y) f(x,y)\, dx\, dy\le
C_1\|\varphi\|_\infty^\alpha\|\partial_x\varphi\|_\infty^{1-\alpha}.
$$
Thus, the function $f$ is
Nikolskii--Besov $\alpha$-smooth along $e_1$ by Theorem \ref{T1,5.1}.
}
\end{example}

\section{Nikolskii--Besov classes with respect to Gaussian measures on $\mathbb{R}^n$}\label{sect5}

In this section we introduce Gaussian Nikolskii--Besov classes
on finite-dimensional spaces. The principal difference with the previously defined classes is that
now we deal with Gaussian measures in place of Lebesgue measure.
The effect of this is some dimension-free estimates which enable us to extend the basic constructions
to the infinite-dimensional case (which is done is the next section).
For these classes  we obtain an equivalent description
in terms of the Ornstein--Uhlenbeck semigroup in case $p>1$, which generalizes
Theorem~\ref{T1.2}. We also find the rate of approximation of functions in Gaussian
Nikolskii--Besov classes by the Ornstein--Uhlenbeck semigroup.
This estimate yields a new Poincar\'e-type inequality.

Let us recall some known facts and notation.
Let $\gamma$ be the standard Gaussian measure on~$\mathbb{R}^n$,
i.e., the measure with density
$$
(2\pi)^{-n/2}\exp(-|x|^2/2)
$$
with respect to the standard Lebesgue measure on $\mathbb{R}^n$.
Let now $\{T_t\}_{t\ge0}$ be the Ornstein--Uhlenbeck semigroup defined by
$$
T_tf(x) := \int_{\mathbb{R}^n} f(e^{-t}x+\sqrt{1-e^{-2t}}y)\, \gamma(dy),\quad f\in L^1(\gamma).
$$
Note that for any function $\varphi\in C_0^\infty(\mathbb{R}^n)$ one has
$$
T_t\varphi(x) = (2\pi(1-e^{-2t}))^{-n/2}\int_{\mathbb{R}^n}\varphi(z)
\exp\Bigl(\frac{-|z-e^{-t}x|^2}{2(1-e^{-2t})}\Bigr)\, dz,
$$
which is readily verified by the change of variables.
We also note that
\begin{equation}\label{eq1}
\nabla T_t\varphi(x) = \frac{e^{-t}}{\sqrt{1-e^{-2t}}}\int \varphi(e^{-t}x+\sqrt{1-e^{-2t}}y) y \, \gamma(dy).
\end{equation}

Let $L$ be the Ornstein--Uhlenbeck operator defined by
$$
L\varphi(x) = \Delta\varphi(x) - \langle x, \nabla\varphi(x)\rangle,\quad \varphi\in C_0^\infty(\mathbb{R}^n).
$$
By the same expression $L$ is defined on functions that have local second order Sobolev derivatives.
It is known that
$$
\frac{d}{dt}T_t\varphi = LT_t\varphi,\quad T_0\varphi = \varphi.
$$
Set
\begin{equation}\label{cons}
c_t := \int_0^t\frac{e^{-\tau}}{\sqrt{1-e^{-2\tau}}}\ d\tau.
\end{equation}
We note that $c_t\le (2t)^{1/2}$ and $\lim\limits_{t\to\infty}c_t = \pi/2$.

Let $\mathrm{Lip}_1(\mathbb{R}^n)$ be the class of all $1$-Lipschitz functions on $\mathbb{R}^n$.
The Kantorovich norm associated with the  measure $\gamma$ is defined by
$$
\|f\|_{\rm K,\gamma} := \sup\biggl\{\int fg\, d\gamma,\ g\in \mathrm{Lip}_1(\mathbb{R}^n)\biggr\}
$$
on the subspace of functions $f\in L^1(\gamma)$ with zero integral
for which this quantity is finite. In particular, this norm is finite
on functions in $W^{1,1}(\gamma)$ with zero integral (see, e.g., \cite[Lemma~4.5]{B14}).

In this section we also use the following notation (which in the previous sections
was employed for Lebesgue measure):
$$
\|f\|_p := \biggl(\int_{\mathbb{R}^n} |f(x)|^p\, \gamma(dx)\biggr)^{1/p},\quad p\in[1,\infty).
$$
For a mapping $\Phi=(\Phi_i)\in C_0^\infty(\mathbb{R}^n, \mathbb{R}^n)$ set
$$
\mathrm{div}_\gamma\Phi = \sum_{i=1}^n (\partial_{x_i}\Phi_i - x_i\Phi_i).
$$

The following definition is inspired by Definition \ref{D1.1}.

\begin{definition}\label{D2.1} Let $\alpha\in (0, 1]$.
A function $f\in L^p(\gamma)$ belongs to the Gaussian Nikolskii--Besov
class $B^\alpha_p(\gamma)$ if there is a number $C$ such that for every
mapping $\Phi\in C_0^\infty(\mathbb{R}^n, \mathbb{R}^n)$ we have
$$
\int_{\mathbb{R}^n}
 f \mathrm{div}_\gamma\Phi \, d\gamma \le C\|\Phi\|^\alpha_q\|\mathrm{div}_\gamma\Phi\|^{1-\alpha}_q,
$$
where $1/p + 1/q = 1$.
Let $V_\gamma^{p, \alpha}(f)$ be the infimum of such numbers $C$.
\end{definition}

By approximation
 the above estimate extends from vector fields $\Phi\in~C_0^\infty(\mathbb{R}^n, \mathbb{R}^n)$
to vector fields $\Phi\in C^\infty(\mathbb{R}^n, \mathbb{R}^n)$
such that $\Phi_i \in L^q(\gamma)$, $\mathrm{div}_\gamma\Phi \in L^q(\gamma)$.

We also introduce an analog of the quantity from Definition \ref{D1.3}.

\begin{definition}\label{D2.2}
For a function $f\in L^p(\gamma)$ set
$$
U^{p, \alpha}_\gamma(f) := \sup_{t>0} t^{\frac{1-\alpha}{2}}\|\nabla T_tf\|_p,
$$
where $\{T_t\}_{t\ge0}$ is the Ornstein--Uhlenbeck semigroup.
\end{definition}

Let us recall (see, e.g., \cite[Proposition 5.4.8]{GM})
that for a function $f\in L^p(\gamma)$ with $p>1$ we have
$\nabla T_tf\in L^p(\gamma)$ for any  $t>0$,
hence $T_tf\in W^{p,1}(\gamma)$, where $W^{p,1}(\gamma)$ is the Sobolev class
with respect to~$\gamma$, defined similarly to the usual Sobolev classes by replacing Lebesgue measure
with~$\gamma$ (see \cite{GM}, \cite{DM}, \cite{B14}, and~\cite{Shig}).
Certainly, in case $p=1$, we set  $U^{1, \alpha}_\gamma(f)=\infty$ if
$T_tf\not\in W^{1,1}(\gamma)$ for some $t>0$.

\begin{lemma}\label{lem2.1}
For any function $\varphi\in C_0^\infty(\mathbb{R}^n)$ one has
$$
\|\nabla T_t\varphi\|_q\le C\bigl(q/(q-1)\bigr)\frac{e^{-t}}{\sqrt{1-e^{-2t}}}\|\varphi\|_q
\quad \forall\, q\in(1, \infty],
$$
and for any $\Phi\in C_0^\infty(\mathbb{R}^n, \mathbb{R}^n)$
one has
$$
\|\mathrm{div}_\gamma T_t\Phi\|_p\le C(p)(1-e^{-2t})^{-1/2}\|\Phi\|_p
\quad \forall\, p\in[1, \infty),
$$
 where
$$
C(p) := \biggl((2\pi)^{-1/2}\int_{\mathbb{R}}|s|^pe^{-\frac{s^2}{2}}\, ds\biggr)^{1/p}.
$$
\end{lemma}
\begin{proof}
For any vector field $\Phi\in C_0^\infty(\mathbb{R}^n, \mathbb{R}^n)$
we have (omitting the indication of $\mathbb{R}^n$ in the limits of integration below)
$$
\mathrm{div}_\gamma T_t\Phi(x) =
(1-e^{-2t})^{-1/2}\int \langle\Phi(e^{-t}x+\sqrt{1-e^{-2t}}y), e^{-t}y - \sqrt{1-e^{-2t}}x\rangle\, \gamma(dy).
$$
Thus, whenever $p\in[1, \infty)$, we obtain
\begin{align*}
\|\mathrm{div}_\gamma  & T_t\Phi\|_p
\\
&=
(1-e^{-2t})^{-1/2}\biggl(\int \biggl|\int\langle\Phi(e^{-t}x+\sqrt{1-e^{-2t}}y), e^{-t}y -
\sqrt{1-e^{-2t}}x\rangle\gamma(dy)\biggr|^p\, \gamma(dx)\biggr)^{1/p}
\\
&\le
(1-e^{-2t})^{-1/2}\biggl(\int \int \Bigl|\langle\Phi(e^{-t}x
+\sqrt{1-e^{-2t}}y), e^{-t}y - \sqrt{1-e^{-2t}}x\rangle\Bigr|^p\, \gamma(dy)\,\gamma(dx)\biggr)^{1/p}
\\
&=
(1-e^{-2t})^{-1/2}\biggl(\int \int|\langle\Phi(x), y\rangle|^p\, \gamma(dy)\, \gamma(dx)\biggr)^{1/p}
\\
&=
(1-e^{-2t})^{-1/2}\biggl((2\pi)^{-1/2}
\int_{\mathbb{R}}|s|^pe^{-\frac{s^2}{2}}\, ds\biggr)^{1/p}\biggl(\int |\Phi(x)|^p\, \gamma(dx)\biggr)^{1/p}
\\
&=
(1-e^{-2t})^{-1/2}\|\Phi\|_p\biggl((2\pi)^{-1/2}\int_{\mathbb{R}}|s|^pe^{-\frac{s^2}{2}}\, ds\biggr)^{1/p}=
C(p)(1-e^{-2t})^{-1/2}\|\Phi\|_p.
\end{align*}
Hence the second part of the lemma is proved.

For proving the first part of the lemma we observe that, letting $p=q/(q-1)$, one has
$$
\|\nabla T_t\varphi\|_q =
\sup_{\substack{\Phi\in C_0^\infty\\ \|\Phi\|_p \le1}}
\int\langle\Phi, \nabla T_t\varphi\rangle \, d\gamma=
\sup_{\substack{\Phi\in C_0^\infty\\ \|\Phi\|_p \le1}}
-e^{-t}\int\varphi{\rm div}_\gamma T_t\Phi\, d\gamma.
$$
Applying the already obtained estimate for the divergence ${\rm div}_\gamma T_t\Phi$
in case $q\in(1, \infty]$ we obtain the announced estimate
$$
\|\nabla T_t\varphi\|_q\le C(p)\frac{e^{-t}}{\sqrt{1-e^{-2t}}}\|\varphi\|_q,
$$
which completes the proof.
\end{proof}

We now proceed to the first main result of this section.

\begin{theorem}\label{T2.1}
For any function $f\in B^\alpha_p(\gamma)$, where $p\in[1, \infty)$, we have
$$
\|f - T_tf\|_p\le  2^{1-\alpha}C(p)^\alpha c_t^\alpha\ V_\gamma^{p,\alpha}(f),
$$
where $c_t$ is defined by {\rm(\ref{cons})} and
$$
C(p) := \biggl((2\pi)^{-1/2}\int_{\mathbb{R}}|s|^pe^{-\frac{s^2}{2}}\, ds\biggr)^{1/p}.
$$
\end{theorem}
\begin{proof}
Let $q=p/(p-1)$ and let $\varphi\in C_0^\infty(\mathbb{R}^n)$ be such that $\|\varphi\|_q\le1$.
We have
\begin{multline*}
\int (T_tf - f)\varphi\, d\gamma
= \int (T_t\varphi -\varphi)fd\gamma = \int \biggl[\int_0^t\frac{\partial}{\partial\tau}T_\tau\varphi d\tau\biggr] f d\gamma \\
=
\int \biggl[\int_0^t LT_\tau\varphi d\tau\biggr] f d\gamma = \int \biggl[\int_0^t \mathrm{div}_\gamma\nabla T_\tau\varphi d\tau\biggr] f d\gamma =
\int \mathrm{div}_\gamma\biggl[\int_0^t\nabla T_\tau\varphi d\tau\biggr] f d\gamma.
\end{multline*}
Note that
$$
x\mapsto\int_0^t\nabla T_\tau\varphi(x) d\tau\in C^\infty_b(\mathbb{R}^n, \mathbb{R}^n)
\quad \forall\in t\ge0.
$$
Thus, by the definition of the space $B_p^\alpha(\gamma)$ we have (again omitting the indication of $\mathbb{R}^n$
in the limits of integration)
\begin{multline*}
\int \mathrm{div}_\gamma\biggl[\int_0^t\nabla T_\tau\varphi \, d\tau\biggr] f \, d\gamma
\\
\le
V_\gamma^{p,\alpha}(f)\biggl(\int \biggl|\int_0^t\nabla T_\tau\varphi \, d\tau\biggr|^q\, d\gamma\biggr)^{\alpha/q}
\biggl(\int \biggl|\mathrm{div}_\gamma\int_0^t\nabla T_\tau\varphi\, d\tau\biggr|^q\, d\gamma\biggr)^{(1-\alpha)/q}.
\end{multline*}
We now estimate each term on the right-hand side separately.
Let us begin with the second term:
$$
\biggl(\int \biggl|\mathrm{div}_\gamma\int_0^t\nabla T_\tau\varphi \, d\tau\biggr|^q\, d\gamma\biggr)^{1/q}=
\biggl(\int |T_t\varphi -\varphi|^q\, d\gamma\biggr)^{1/q}\le \|T_t\varphi\|_q+\|\varphi\|_q\le2.
$$
Let us now consider the first term.
By Lemma \ref{lem2.1} one has
$$
\|\nabla T_\tau\varphi\|_q\le
C(p)\frac{e^{-\tau}}{\sqrt{1-e^{-2\tau}}}\|\varphi\|_q\le C(p)\frac{e^{-\tau}}{\sqrt{1-e^{-2\tau}}}.
$$
Hence, applying again Minkowski's inequality for integrals, we have
\begin{align*}
\biggl(\int \biggl|\int_0^t\nabla T_\tau\varphi d\tau\biggr|^q\, d\gamma\biggr)^{1/q}
&\le
\biggl(\int \biggl(\int_0^t|\nabla T_\tau\varphi| \, d\tau\biggr)^q\, d\gamma\biggr)^{1/q}\\
&\le
\int_0^t\biggl(\int |\nabla T_\tau\varphi|^q\, d\gamma\biggr)^{1/q}\, d\tau
\le
C(p)\int_0^t\frac{e^{-\tau}}{\sqrt{1-e^{-2\tau}}}\, d\tau = C(p)c_t.
\end{align*}
Thus, we arrive at the estimate
$$
\int (T_tf - f)\varphi \, d\gamma\le 2^{1-\alpha}C(p)^\alpha c_t^\alpha\ V_\gamma^{p,\alpha}(f).
$$
Taking the supremum over functions $\varphi\in C_0^\infty(\mathbb{R}^n)$ with $\|\varphi\|_q\le1$
we obtain the announced bound.
\end{proof}

By passing to the limit $t\to\infty$ in the
previous theorem we obtain the following Poincar\'e-type inequality.
Let $\mathbb{E}f$ denote the integral of $f$ against~$\gamma$.

\begin{corollary}\label{poinc}
For any function $f\in B_p^\alpha(\gamma)$ with $p\in[1, \infty)$, we have
$$
\|f - \mathbb{E}f\|_p\le  2^{1-2\alpha}\pi^{\alpha}C(p)^\alpha\ V_\gamma^{p,\alpha}(f).
$$
\end{corollary}

We now obtain an
analog of the fractional Hardy--Landau--Littlewood inequality with respect to a Gaussian measure,
which generalizes a similar estimate for $\alpha=1$ from~\cite{BWSH}.
Recall that the classical Hardy--Landau--Littlewood inequality on the real line with Lebesgue measure
states that $\|\varphi'\|_1^2\le 2\|\varphi\|_1\|\varphi''\|_1$, which can be rewritten as
$\|f\|_1^2\le 2 \|f'\|_1 \|f\|_{{\rm K}}$, where $\|f\|_{{\rm K}}$
for $f$ with zero integral is the Kantorovich norm  defined as the supremum
of integrals of $fg$ over $1$-Lipschitz functions~$g$.

\begin{theorem}\label{T2.3}
For any function $f\in B^\alpha_1(\gamma)$ with zero integral we have
$$
\|f\|_1\le  3(V_\gamma^{1,\alpha}(f))^{1/(1+\alpha)}\|f\|^{\alpha/(1+\alpha)}_{\rm K,\gamma}.
$$
\end{theorem}
\begin{proof} Let us fix a number $t\in (0, \infty)$.
By the triangle inequality and Theorem \ref{T2.1} we have
$$
\|f\|_1\le \|T_tf\|_1 + \|f-T_tf\|_1\le \|T_tf\|_1 + 2V_\gamma^{1,\alpha}(f)t^{\alpha/2}.
$$
We need to estimate $\|T_tf\|_1$. For every function $\varphi\in C_0^\infty(\mathbb{R}^n)$
with $\|\varphi\|_\infty\le1$ we have by the definition of the Kantorovich norm
$$
\int \varphi T_tf\, d\gamma = \int T_t\varphi f\, d\gamma\le \|\nabla T_t\varphi\|_\infty\|f\|_{\rm K, \gamma}.
$$
Since $\|\varphi\|_\infty\le1$, by Lemma \ref{lem2.1} the last expression is not greater than
$$
\frac{e^{-t}}{\sqrt{1-e^{-2t}}}\|f\|_{\rm K, \gamma}\le t^{-1/2}\|f\|_{\rm K, \gamma}.
$$
Thus,
$$
\|f\|_1 \le t^{-1/2}\|f\|_{\rm K, \gamma} + 2V_\gamma^{1,\alpha}(f)t^{\alpha/2}.
$$
Setting $t^{1/2}= (\|f\|_{\rm K, \gamma}/V_\gamma^{1,\alpha}(f))^{1/(1+\alpha)}$ we obtain
$$
\|f\|_1 \le 3(V_\gamma^{1,\alpha}(f))^{1/(1+\alpha)}\|f\|_{\rm K, \gamma}^{\alpha/(1+\alpha)},
$$
which completes the proof.
\end{proof}

The next lemma is an analog of Theorem \ref{T2.1} with
the quantity $V_\gamma^{p,\alpha}(f)$  replaced by~$U^{p, \alpha}_\gamma(f)$.

\begin{lemma}\label{lem2.2}
Let $f\in L^p(\gamma)$ with $p\in[1,\infty)$ be such that $U^{p, \alpha}_\gamma(f)<\infty$.
Then
$$
\|f-T_tf\|_p\le 4C(p)\alpha^{-1}t^{\alpha/2}U^{p, \alpha}_\gamma(f),
$$
where
$$
C(p) := \biggl((2\pi)^{-1/2}\int_{\mathbb{R}}|s|^pe^{-\frac{s^2}{2}}ds\biggr)^{1/p}.
$$
\end{lemma}
\begin{proof}
For any function $\varphi\in C_0^\infty(\mathbb{R}^n)$ we have
\begin{multline*}
\int \varphi (T_tf -f)\, d\gamma
=
\int (T_t\varphi -\varphi)f\, d\gamma
=
\int \biggl[\int_0^t\frac{\partial}{\partial\tau}T_\tau\varphi \, d\tau\biggr] f\, d\gamma
\\
=
\int \biggl[\int_0^t LT_\tau\varphi\, d\tau\biggr] f\, d\gamma
= \int \biggl[\int_0^t \mathrm{div}_\gamma\nabla T_\tau\varphi \, d\tau\biggr] f \, d\gamma =
\int_0^t\int T_{\tau/2}\mathrm{div}_\gamma\nabla T_{\tau/2}\varphi f \, d\gamma \, d\tau,
\end{multline*}
where the last equality is due to the known identity $LT_t u = T_t L u$
for functions $u\in C_b^\infty(\mathbb{R}^n)$.
The gradient $\nabla T_{\tau/2}\varphi$ is estimated by Lemma \ref{lem2.1} as follows:
$$
\|\nabla T_{\tau/2}\varphi\|_q\le C(p)\frac{e^{-(\tau/2)}}{\sqrt{1-e^{-\tau}}}\|\varphi\|_q.
$$
Thus,
\begin{align*}
\int \varphi (T_tf -f)\, d\gamma &=
\int_0^t\int T_{\tau/2}\mathrm{div}_\gamma\nabla T_{\tau/2}\varphi f \, d\gamma \, d\tau
\\
&=
\int_0^t\int\mathrm{div}_\gamma\nabla T_{\tau/2}\varphi T_{\tau/2}f \, d\gamma \, d\tau
=
-\int_0^t\int\langle\nabla T_{\tau/2}\varphi, \nabla T_{\tau/2}f\rangle\, d\gamma\, d\tau
\\
&\le
C(p)\|\varphi\|_qU^{p, \alpha}_\gamma(f)
\int_0^t (\tau/2)^{(\alpha-1)/2} \frac{e^{-(\tau/2)}}{\sqrt{1-e^{-\tau}}}
\, d\tau
\\
&\le
C(p)\|\varphi\|_qU^{p, \alpha}_\gamma(f)2\int_0^{t/2} s^{(\alpha-1)/2} s^{-1/2}\, d\tau
=
C(p)4\alpha^{-1}(t/2)^{\alpha/2}\|\varphi\|_qU^{p, \alpha}_\gamma(f).
\end{align*}
Taking the supremum over $\varphi\in C_0^\infty(\mathbb{R}^n)$
and noting that $2^{-\alpha/2}<1$
we obtain the announced estimate.
\end{proof}

Let us proceed to a characterization of the Gaussian Nikolskii--Besov classes
in terms of the behavior  of the Ornstein--Uhlenbeck semigroup near zero.

\begin{theorem}\label{T2.2}
Let $f\in L^p(\gamma)$, $p\in(1, \infty)$.
Then $V^{p,\alpha}_\gamma(f)<\infty$ if and only if
$U^{p,\alpha}_\gamma(f)<\infty$.

Moreover, one has
$$
U^{p,\alpha}_\gamma(f)\le C(q)^{1-\alpha}V^{p, \alpha}_\gamma(f),
\quad
V^{p, \alpha}_\gamma(f)\le(4C(p)\alpha^{-1}+1)U^{p, \alpha}_\gamma(f),
$$
where $1/p + 1/q =1$ and
$$
C(p) := \biggl((2\pi)^{-1/2}\int_{\mathbb{R}}|s|^pe^{-\frac{s^2}{2}}\, ds\biggr)^{1/p}.
$$
\end{theorem}
\begin{proof}
Assume that $V^{p,\alpha}_\gamma(f)<\infty$.
For every vector field  $\Phi\in C_0^\infty(\mathbb{R}^n, \mathbb{R}^n)$
with $\|\Phi\|_q\le 1$ one can easily verify that
$$
\int \langle \Phi, \nabla T_tf\rangle \, d\gamma =
-e^{-t}\int {\rm div_\gamma}\bigl(T_t\Phi\bigr)f\, d\gamma.
$$
The previous expression is estimated from above by
$$
e^{-t}V^{p, \alpha}_\gamma(f)\|T_t\Phi\|^\alpha_q\|\mathrm{div}_\gamma T_t\Phi\|^{1-\alpha}_q.
$$
It is known (and easily verified)  that
$$
\|T_t\Phi\|_q
\le \|\Phi\|_q.
$$
Since $q\in[1,\infty)$, by Lemma \ref{lem2.1} we have
$$
\|\mathrm{div}_\gamma T_t\Phi\|_q \le C(q)(1-e^{-2t})^{-1/2}\|\Phi\|_q.
$$
Summing up these estimates we obtain
$$
\int \langle \Phi, \nabla T_tf\rangle\,
 d\gamma \le C(q)^{1-\alpha}V^{p, \alpha}_\gamma(f)e^{-t}(1-e^{-2t})^{-(1-\alpha)/2}
\le C(q)^{1-\alpha}V^{p, \alpha}_\gamma(f)(2t)^{-(1-\alpha)/2}.
$$
Taking the supremum over $\Phi$ we obtain the desired bound.

Let us now assume that $U^{p,\alpha}_\gamma(f)<\infty$.
Let $\Phi\in C_0^\infty(\mathbb{R}^n, \mathbb{R}^n)$.
We have
$$
\int{\rm div}_\gamma\Phi f \, d\gamma =
\int {\rm div}_\gamma\Phi (f - T_tf)\, d\gamma + \int {\rm div}_\gamma\Phi T_tf\, d\gamma.
$$
For the first term on the right Lemma \ref{lem2.2} gives
$$
\int{\rm div}_\gamma\Phi (f - T_tf)\, d\gamma\le\|{\rm div}_\gamma\Phi\|_q\|f - T_tf\|_p\le
4C(p)\alpha^{-1}t^{\alpha/2} U^{p, \alpha}_\gamma(f) \|{\rm div}_\gamma\Phi\|_q.
$$
For the second term we have
$$
\int {\rm div}_\gamma\Phi T_tf\, d\gamma = - \int \langle\Phi, \nabla T_tf\rangle\, d\gamma\le
\|\Phi\|_q\|\nabla T_tf\|_p
\le U^{p,\alpha}_\gamma(f)t^{-(1-\alpha)/2}\|\Phi\|_q.
$$
Thus,
$$
\int {\rm div}_\gamma\Phi f\, d\gamma \le
4C(p)\alpha^{-1}t^{\alpha/2} U^{p, \alpha}_\gamma(f) \|{\rm div}_\gamma\Phi\|_q
+ U^{p,\alpha}_\gamma(f)t^{-(1-\alpha)/2}\|\Phi\|_q.
$$
Taking $t = \|\Phi\|_q^2\|{\rm div}_\gamma\Phi\|_q^{-2}$
we find that
$$
\int {\rm div}_\gamma\Phi f \, d\gamma \le (4C(p)\alpha^{-1}+1)
U^{p, \alpha}_\gamma(f)\|\Phi\|^\alpha_q\|{\rm div}_\gamma\Phi\|^{1-\alpha}_q,
$$
which completes the proof.
\end{proof}

It is worth noting that the implication
$U^{p,\alpha}_\gamma(f)<\infty \Rightarrow V^{p,\alpha}_\gamma(f)<\infty$
is also true for $p=1$ (with the same proof).

\section{Nikolskii--Besov classes with respect
to Gaussian measures on infinite-dimensional spaces}\label{sect6}

We now proceed to the infinite-dimensional case.
Let $X$ be a real Hausdorff locally convex space
with  the topological dual space $X^*$.

Let $\mathcal{FC}^\infty(X)$ be the set of all functions
$\varphi$ on $X$ of the form $\varphi(x) = \psi(l_1(x), \ldots, l_n(x))$,
where $\psi\in C_b^\infty(\mathbb{R}^n)$, $l_i\in X^*$, and let
$\mathcal{FC}_0^\infty(X)$ be its subclass consisting of functions for which
 $\psi$ can be chosen in~$C_0^\infty(\mathbb{R}^n)$.

We observe that $\mathcal{FC}_0^\infty(X)$ is not a linear space (unlike~$\mathcal{FC}^\infty(X)$),
since a nonzero function of class $C_0^\infty(\mathbb{R}^n)$ does not have compact support
as a  function on~$\mathbb{R}^{n+1}$. However, in some cases the class $\mathcal{FC}_0^\infty(X)$ is
more convenient.

Let $\gamma$
be a Radon centered Gaussian measure on~$X$, i.e., it is a probability measure on the Borel $\sigma$-algebra
of~$X$ such that for every Borel set $B\subset X$ the value $\gamma(B)$ equals the supremum
of $\gamma(K)$ over compact sets $K\subset B$, and, in addition, every continuous linear functional $l$
on $X$ is a centered Gaussian random variable on~$(X,\gamma)$.  The latter means that
the induced measure $\gamma\circ l^{-1}$ is either Dirac's measure at zero or has a density
of the form $c_1\exp (-c_2t^2)$.

For any function $f\in L^p(\gamma)$ we set
$$
\|f\|_p := \|f\|_{L^p(\gamma)} := \biggl(\int_{X} |f|^p\, d\gamma\biggr)^{1/p}.
$$

Let $H\subset X$ be the Cameron--Martin space of the measure $\gamma$, i.e.,
the space of all vectors $h$ such that $\gamma_h \sim \gamma$, where $\gamma_h(B)=\gamma(B-h)$.
If $\gamma$ is the countable power of the standard Gaussian measure on the real line and is regarded on the
space $\mathbb{R}^\infty$ of all real sequences,
then $H$ is the standard  Hilbert space~$l^2$
(for the standard Gaussian measure on $\mathbb{R}^n$
the Cameron--Martin space is $\mathbb{R}^n$ itself).
For a general Radon centered Gaussian measure, the Cameron--Martin space
is also a separable Hilbert space (see \cite[Theorem 3.2.7 and Proposition 2.4.6]{GM})
with the inner product $\langle\cdot,\cdot\rangle_H$ and norm $|\cdot|_H$ defined
by
$$
|h|_H=\sup \biggl\{ l(h)\colon\, \int_X l^2\, d\gamma \le 1, \ l\in X^{*}\biggr\}.
$$
Let $\{l_i\}_{i=1}^\infty\subset X^*$ be an  orthonormal basis in the closure $X_\gamma^*$
 of the set $X^*$ in $L^2(\gamma)$.
 There is an orthonormal basis $\{e_i\}_{i=1}^\infty$ in $H$ such that $l_i(e_j) = \delta_{i,j}$.

For every vector $h\in H$ there is a unique element $\widehat{h}\in X_\gamma^*$ such that
$$
l(h)=\int_X l(x)\widehat{h}(x)\, \gamma(dx)\quad \forall\, l\in X^*.
$$
According to the Cameron--Martin formula, for every $h\in H$ the shifted measure
$\gamma(\cdot-h)$ has density $\exp(\widehat{h}-|h|_H^2/2)$ with respect to~$\gamma$.
It follows that the measure $\gamma$ is Fomin differentiable along $h$
and its logarithmic derivative is $-\widehat{h}$, i.e.,
the following integration by parts formula holds:
$$
\int_X \partial_h f(x)\, \gamma(x)=\int_X f(x)\widehat{h}(x)\, \gamma(dx), \quad f\in \mathcal{FC}^\infty(X),
$$
where $\partial_h f(x)=\lim\limits_{t\to 0} t^{-1}(f(x+th)-f(x))$, as in the finite-dimensional case.

We shall use below that for any orthonormal family $l_1,\ldots,l_n\in X_\gamma^*$
the distribution of the vector $(l_1, \ldots, l_n)$, i.e., the image of~$\gamma$,
 is the standard Gaussian measure $\gamma_n$ on $\mathbb{R}^n$.

Let $\mathcal{FC}^\infty(X, H)$ be the set of all vector fields $\Phi$ of the form
$$
\Phi(x) = \sum_{i=1}^n\Psi_i(g_1(x), \ldots, g_n(x))h_i,
$$
where $\Psi_i\in C_b^\infty(\mathbb{R}^n)$, $g_i\in X^*$, $h_i\in H$.
Let $\mathcal{FC}_0^\infty(X, H)$ be the subset of this class consisting of mappings for which
$\Psi_i$ can be chosen with compact support. In this representation we can always take vectors
$h_i$ orthogonal in~$H$ and functionals $g_i$ orthogonal in~$X_\gamma^*$ such that $g_i(h_j)=\delta_{ij}$.

Given $\{l_i\}$ and $\{e_i\}$ as above,
for $\varphi\in\mathcal{FC}^\infty(X)$ of the form $\varphi(x) = \psi(l_1(x), \ldots, l_n(x))$ set
$$
\nabla\varphi(x) = \sum_{j=1}^n\partial_{x_j}\psi (l_1(x), \ldots, l_n(x)) e_j.
$$
Let ${\rm div}_\gamma$ be the ``adjoint operator'' to the gradient operator $\nabla$
with respect to~$\gamma$, i.e.,
for any vector field $\Phi\in \mathcal{FC}^\infty(X, H)$ and any function
$\phi\in \mathcal{FC}^\infty(X)$ one has
$$
\int_X ({\rm div}_\gamma\Phi) \varphi\, d\gamma = -\int_X \langle\Phi, \nabla\varphi\rangle_H\, d\gamma.
$$
It is readily verified that for a vector field $\Phi\in\mathcal{FC}^\infty(X, H)$ of the form
$$
\Phi(x) = \sum_{i=1}^n\Psi_i(l_1(x), \ldots, l_n(x))e_i
$$
with biorthogonal $\{l_i\}$ and $\{e_i\}$ as above one has
$$
{\rm div}_\gamma\Phi (x) = \sum_{j=1}^n\partial_{x_j}\Psi_j(l_1(x), \ldots, l_n(x)) -
l_j(x)\Psi_j(l_1(x), \ldots, l_n(x)).
$$
For any vector field $\Phi\in\mathcal{FC}_0^\infty(X, H)$
its divergence ${\rm div}_\gamma\Phi$ is a bounded function.

Denote by $L^p(\gamma, H)$ the space of all $\gamma$-measurable vector fields $F$ with values in $H$
such that $|F|_H^p$ is $\gamma$-integrable.
Let $W^{p, 1}(\gamma)$ be the closure of $\mathcal{FC}^\infty(X)$ with respect to the natural Sobolev
norm
$$
\|f\|_{p,1}=\|f\|_p+\|\nabla f\|_p.
$$
This class coincides with
the space of all functions $f$ from $L^p(\gamma)$ such that there is a
mapping $F\in L^p(\gamma, H)$, denoted by the symbol $\nabla f$,
with the property
$$
\int_X \langle\Phi, F\rangle_H \, d\gamma = - \int_X f {\rm div}_\gamma\Phi \, d\gamma
\quad
\forall\, \Phi\in \mathcal{FC}^\infty(X, H).
$$
The gradient in $W^{p, 1}(\gamma)$ is an extension of the $H$-gradient for
functions in~$\mathcal{FC}^\infty(X)$. On Sobolev classes over Gaussian measures,
see \cite{GM}, \cite{DM}, \cite{B14G}, and \cite{B14}.

The action of the Ornstein--Uhlenbeck semigroup on a function $f\in L^1(\gamma)$ is defined by the same equality
as in the finite-dimensional case:
$$
T_tf(x) := \int_{X} f(e^{-t}x+\sqrt{1-e^{-2t}}y)\, \gamma(dy).
$$
The Kantorovich norm  associated with  $\gamma$
is defined  by
$$
\|f\|_{\rm K,\gamma} := \sup\biggl\{\int_X \varphi f\, d\gamma\colon
\varphi\in \mathcal{FC}^\infty(X), \  |\nabla\varphi|_H\le1\biggr\}
$$
on functions $f\in L^1(\gamma)$ with vanishing integral for which this norm is finite.
Actually, this is the restriction of the Kantorovich norm generated by the subspace $H$
on the space of signed measures with zero value on~$X$ integrating $H$-Lipschitz functions
(on such norms, see, e.g.,~\cite{BKol}).

The Kantorovich norm extends naturally to the space of all $\gamma$-integrable functions
$f$ such that $\|f-\mathbb{E} f\|_{\rm K,\gamma}<\infty$, where $\mathbb{E} f$ is the integral of~$f$, by setting
$$
\|f\|_{\rm K,\gamma} :=\|f-\mathbb{E} f\|_{\rm K,\gamma}+|\mathbb{E} f|.
$$

It is known  (see \cite{GM}, \cite{DM}) that
for every function $\varphi$ with $|\nabla\varphi|_H\le1$ the function
$\exp( c|\varphi|^2)$ is $\gamma$-integrable, hence
$\|f\|_{\rm K,\gamma}<\infty$ provided that $f\sqrt{|\ln |f||}\in L^1(\gamma)$.
Therefore, this norm is finite on the Sobolev space $W^{1,1}(\gamma)$ and, more generally,
on the class $BV(\gamma)$ (see, e.g., \cite{FH} or~\cite{Led}).
Applying \cite[Proposition~5.22]{Led-b} we see that
$\|f\|_{\rm K,\gamma}\le \|\nabla f\|_{L^1(\gamma)}$.

Given an orthonormal basis $\{l_n\}\subset X^*$ in $X_\gamma^*$,
for any function $f\in L^1(\gamma)$ let $\mathbb{E}_nf$ be a function on $\mathbb{R}^n$ such that
$$
\int_{\mathbb{R}^n} \psi\mathbb{E}_nf\, d\gamma_n =
\int_X \psi\bigl(l_1(x), \ldots, l_n(x)\bigr) f(x)\, \gamma(dx)
\quad
\forall\, \psi\in C_b^\infty(\mathbb{R}^n),
$$
 where $\gamma_n$ is the standard Gaussian measure on $\mathbb{R}^n$.
In other words, $\mathbb{E}_nf(l_1,\ldots,l_n)$ is the conditional expectation of $f$ with respect to the
$\sigma$-algebra generated by $l_1,\ldots,l_n$.

Similarly, for a vector field $F\in L^1(\gamma, H)$ set
$$
\mathbb{E}_nF = (\mathbb{E}_nF_i)_{i=1}^n, \quad F_i:=\langle F, e_i\rangle_H.
$$
By the known property of conditional expectations, for any function $f\in L^p(\gamma)$,
as $n\to\infty$, we have
$$
\|f - \mathbb{E}_nf(l_1, \ldots, l_n)\|_p\to 0
$$
 and for any mapping $F\in L^p(\gamma, H)$ we have
$$
\|F - \mathbb{E}_nF(l_1, \ldots, l_n)\|_{L^p(\gamma, H)}\to0.
$$

We now give natural analogs of Definition \ref{D2.1} and Definition \ref{D2.2} in the infinite-dimensional case.

\begin{definition}\label{D3.1} Let $\alpha\in (0, 1]$.
A function $f\in L^p(\gamma)$ belongs to the Gaussian Nikolskii--Besov class $B^\alpha_p(\gamma)$
if there is a number $C$ such that for every
vector field $\Phi\in\mathcal{FC}_0^\infty(X, H)$
we have
$$
\int_X f \mathrm{div}_\gamma\Phi \, d\gamma \le C\|\Phi\|^\alpha_q\|\mathrm{div}_\gamma\Phi\|^{1-\alpha}_q,
$$
where $1/p + 1/q = 1$.
Let $V_\gamma^{p, \alpha}(f)$ be the infimum of such numbers $C$.
\end{definition}

It is readily verified that the same value of $V_\gamma^{p, \alpha}(f)$ will be obtained
if we employ the linear space $\mathcal{FC}^\infty(X, H)$ in place of
the class $\mathcal{FC}_0^\infty(X, H)$ (which is not a linear space, as noted above, but has the
advantage that $\mathrm{div}_\gamma\Phi$ is bounded).
This is done by approximation. The case $p=1$  is special, since $q=\infty$,
we consider approximations
of a mapping $\Phi\in \mathcal{FC}^\infty(X, H)$ with the bounded divergence $\mathrm{div}_\gamma\Phi$
by mappings $\Phi_k\in \mathcal{FC}_0^\infty(X, H)$ such that the divergences
$\mathrm{div}_\gamma\Phi_k$ are uniformly bounded and converge pointwise
to $\mathrm{div}_\gamma\Phi$, $\|\Phi_k\|_\infty\le \|\Phi\|_\infty$ and
$\|\mathrm{div}_\gamma\Phi_k\|_\infty\to \|\mathrm{div}_\gamma\Phi\|_\infty$.
This is possible, since everything reduces to the case of $\mathbb{R}^n$,  where we can
find functions $\zeta_k$ of the form $\zeta_k(x)=\zeta_1(x/k)$, $\zeta_1\in C_0^\infty (\mathbb{R}^n)$,
$\zeta_1=1$ on the unit ball. Then for $\Phi_k:=\zeta_k\Psi$ we have
$$
\mathrm{div}_\gamma\Phi_k=\zeta_k \mathrm{div}_\gamma\Phi +\langle \nabla\zeta_k,\Phi\rangle,
$$
where $|\nabla \zeta_k|\le C/k$.

\begin{definition}\label{D3.2} For a function $f\in L^p(\gamma)$ set
$$
U^{p, \alpha}_\gamma(f) := \sup_{t>0} t^{\frac{1-\alpha}{2}}\|\nabla T_tf\|_p,
$$
where $\{T_t\}_{t\ge0}$ is the Ornstein--Uhlenbeck semigroup.
\end{definition}

We observe that the quantities $V_\gamma^{p, \alpha}(f)$ and $U^{p, \alpha}_\gamma(f)$
do not change if add constants to $f$, hence they can be evaluated for functions with
zero integral.

In the infinite-dimensional case we also have
$T_tf\in W^{p,1}(\gamma)$ for any function $f\in L^p(\gamma)$ with $p>1$ and any $t>0$, so that
$\nabla T_tf\in L^p(\gamma, H)$,
see \cite[Proposition 5.4.8]{GM}.
Certainly, in case $p=1$ in the above definition we set $U^{p, \alpha}_\gamma(f)=\infty$
if $T_tf\not\in W^{1,1}(\gamma)$ for some $t>0$.

The following auxiliary  lemmas enable us  to reduce
the main results of this section to the finite-dimensional case.

Let $\{l_n\}\subset X^*$ be an orthonormal basis  in $X_\gamma^*$ as above.

\begin{lemma}\label{lem3.1}
A function $f\in L^p(\gamma)$ belongs to  $B^\alpha_p(\gamma)$ if and only if, for every~$n$,
the function $\mathbb{E}_nf$ belongs to the class $B^\alpha_p(\gamma_n)$
and $\sup_n V_{\gamma_n}^{p, \alpha}(\mathbb{E}_nf)<\infty$.
Moreover,
$$
V_\gamma^{p, \alpha}(f) = \lim\limits_{n\to\infty} V_{\gamma_n}^{p, \alpha}(\mathbb{E}_nf).
$$
\end{lemma}
\begin{proof}
Suppose that
$f\in B^\alpha_p(\gamma)$. Considering the fields of the form $\Phi(l_1,\ldots,l_n)$ we conclude that
$\mathbb{E}_nf\in B^\alpha_p(\gamma_n)$ and
$$
V_{\gamma_n}^{p, \alpha}(\mathbb{E}_nf)\le V_{\gamma_{n+1}}^{p, \alpha}(\mathbb{E}_{n+1}f)\le  V_\gamma^{p, \alpha}(f).
$$
Hence there is a finite limit
\begin{equation}\label{lem3.1.1}
\lim\limits_{n\to\infty} V_{\gamma_n}^{p, \alpha}(\mathbb{E}_nf)\le V_\gamma^{p, \alpha}(f).
\end{equation}
Moreover, we have the equality in this estimate since a field of the form
 $\Psi(g_1,\ldots,g_n)$ can be approximated by fields depending on finitely many functionals~$l_i$.

Let us now assume that $\mathbb{E}_nf\in B^\alpha_p(\gamma_n)$
and $\sup_n V_{\gamma_n}^{p, \alpha}(\mathbb{E}_nf)=C<\infty$.
Let $\Phi\in\mathcal{FC}^\infty(E, H)$ be of the form
$$
\Phi(x) = \sum_{i=1}^n\Psi_i(g_1(x), \ldots, g_k(x))e_i
$$
with $\Psi\in C_b^\infty(\mathbb{R}^k)$, $g_i\in X^*$. We can approximate each $g_i$ by a sequence
of functionals of the form $\sum_{j=1}^n c_{i,j} l_j$ converging in $L^2(\gamma)$, hence
in all $L^p(\gamma)$. Then we obtain approximations (in $L^p$ with divergences)
of $\Phi$ by vector fields depending on~$l_j$,
but for such fields we obviously have
\begin{align*}
\int_X {\rm div}_{\gamma}\Phi f\, d\gamma
&= \int_X ({\rm div}_{\gamma_n}\Psi)(l_1, \ldots, l_n) f\, d\gamma
=
\int_{\mathbb{R}^n} ({\rm div}_{\gamma_n}\Psi) \mathbb{E}_nf\, d\gamma_n
\\
&\le
C \|\Psi\|^\alpha_{L^q(\gamma_n)}\|\mathrm{div}_{\gamma_n}\Psi\|^{1-\alpha}_{L^q(\gamma_n)}
\le
C\|\Phi\|^\alpha_{L^q(\gamma)}\|\mathrm{div}_{\gamma}\Phi\|^{1-\alpha}_{L^q(\gamma)}.
\end{align*}
Hence $f\in B^\alpha_p(\gamma)$.
\end{proof}

The following commutativity properties of $T_t$ are well-known and easily verified.

\begin{lemma}\label{lem3.2}
For any function $f\in L^p(\gamma)$ with $p\in[1, \infty]$, one has
$$
\mathbb{E}_n (T_tf) = T^n_t \mathbb{E}_nf,
$$
where $\{T_t^n\}_{t\ge0}$ is the Ornstein--Uhlenbeck semigroup on $\mathbb{R}^n$
 with the standard Gaussian measure~$\gamma_n$. In case $p>1$ we also have
 $$
\mathbb{E}_n\nabla T_tf = \nabla T^n_t \mathbb{E}_nf.
$$
\end{lemma}

We now present infinite-dimensional versions of the results of the previous section.

\begin{theorem}\label{T3.1}
For any function $f\in B^\alpha_p(\gamma)$ with $p\in[1, \infty)$, one has
$$
\|f - T_tf\|_p\le  2^{1-\alpha}C(p)^\alpha c_t^\alpha\ V_\gamma^{p,\alpha}(f),
$$
where $c_t$ is defined by {\rm(\ref{cons})} and
$$
C(p) := \biggl((2\pi)^{-1/2}\int_{\mathbb{R}}|s|^pe^{-\frac{s^2}{2}}\, ds\biggr)^{1/p}.
$$
\end{theorem}
\begin{proof}
Since $f\in B^\alpha_p(\gamma)$, by Lemma \ref{lem3.1} we have $\mathbb{E}_nf\in B^\alpha_p(\gamma_n)$.
By Theorem \ref{T2.1} we have
$$
\|\mathbb{E}_nf - T^n_t\mathbb{E}_nf\|_{L^p(\gamma_n)}\le
2^{1-\alpha}C(p)^\alpha c_t^\alpha\ V_{\gamma_n}^{p,\alpha}(\mathbb{E}_nf).
$$
Recall that $\mathbb{E}_nf(l_1, \ldots, l_n) \to f$ in $L^p(\gamma)$,
$T^n_t\mathbb{E}_nf (l_1, \ldots, l_n) = \mathbb{E}_nT_tf(l_1, \ldots, l_n) \to T_tf$ in $L^p(\gamma)$,
and
$V_{\gamma_n}^{p,\alpha}(\mathbb{E}_nf) \to V_\gamma^{p,\alpha}(f)$.
Letting $n\to\infty$, we obtain the stated inequality.
\end{proof}

\begin{corollary}\label{poinc2}
For any function $f\in B_p^\alpha(\gamma)$ with $p\in[1, \infty)$, we have
$$
\|f - \mathbb{E}f\|_p\le  2^{1-2\alpha}\pi^{\alpha}C(p)^\alpha\ V_\gamma^{p,\alpha}(f).
$$
\end{corollary}

\begin{theorem}\label{T3.3}
For any function $f\in B^\alpha_1(\gamma)$ with zero integral we have
$$
\|f\|_1\le  3(V_\gamma^{1,\alpha}(f))^{1/(1+\alpha)}\|f\|^{\alpha/(1+\alpha)}_{\rm K,\gamma}.
$$
\end{theorem}
\begin{proof}
For any function $\varphi\in\mathcal{FC}^\infty(E)$
of the form $\varphi(x) = \psi(l_1(x), \ldots, l_n(x))$,  $\psi\in C_b^\infty(\mathbb{R}^n)$,
such that $|\nabla\varphi|_H\le1$ we have
$$
\int_X \varphi f \, d\gamma = \int_{\mathbb{R}^n} \psi \mathbb{E}_n f \, d\gamma_n
$$
and $|\nabla\varphi(x)|_H = |\nabla\psi(l_1(x), \ldots, l_n(x))|$.
Hence
$\|\mathbb{E}_n f\|_{{\rm K}, \gamma_n}\le \|f\|_{\rm K,\gamma}$.
Since $\|\mathbb{E}_nf\|_{L^1(\gamma_n)}\to \|f\|_1$ and
$V_{\gamma_n}^{1,\alpha}(\mathbb{E}_nf) \to V_\gamma^{1,\alpha}(f)$ (by Lemma \ref{lem3.1})
and the estimate is valid in the finite-dimensional case (Theorem \ref{T2.3}),
we obtain the announced estimate.
\end{proof}

\begin{theorem}\label{T3.2}
A function $f\in L^p(\gamma)$ with $p\in(1, \infty)$
belongs to the Gaussian Nikolskii--Besov class $B^\alpha_p(\gamma)$ if and only if
$U^{p,\alpha}_\gamma(f)<\infty$.

Moreover,
$$
U^{p,\alpha}_\gamma(f)\le C(q)^{1-\alpha}V^{p, \alpha}_\gamma(f), \quad
V^{p, \alpha}_\gamma(f)\le(2C(p)+1)U^{p, \alpha}_\gamma(f),
$$
where $1/p + 1/q =1$ and
$$
C(p) := \biggl((2\pi)^{-1/2}\int_{\mathbb{R}}|s|^pe^{-\frac{s^2}{2}}\, ds\biggr)^{1/p}.
$$
\end{theorem}
\begin{proof}
Let $f\in B^\alpha_p(\gamma)$.
By Lemma \ref{lem3.1} we have $\mathbb{E}_nf\in B^\alpha_p(\gamma_n)$.
By Theorem \ref{T2.2} we have $U^{p, \alpha}_{\gamma_n}(\mathbb{E}_nf)<\infty$ and
$$
U^{p,\alpha}_{\gamma_n}(\mathbb{E}_nf)\le C(q)^{1-\alpha}V^{p, \alpha}_{\gamma_n}(\mathbb{E}_nf).
$$
By Lemma \ref{lem3.1} we have
 $V_{\gamma_n}^{p,\alpha}(\mathbb{E}_nf) \to V_\gamma^{p,\alpha}(f)$
and
$$
t^{(1-\alpha)/2}\|\nabla T_tf\|_p = \lim\limits_{n\to\infty}t^{(1-\alpha)/2}\|\mathbb{E}_n\nabla T_tf\|_{L^p(\gamma_n)}.
$$
Since $\mathbb{E}_n\nabla T_tf=\nabla T^n_t\mathbb{E}_nf$
by Lemma \ref{lem3.2}, we obtain
\begin{align*}
t^{(1-\alpha)/2}\|\nabla T_tf\|_p
&\le \limsup_{n\to\infty}\sup_tt^{(1-\alpha)/2}\|\nabla T^n_t\mathbb{E}_n f\|_{L^p(\gamma_n)}
=\limsup_{n\to\infty}U^{p,\alpha}_{\gamma_n}(\mathbb{E}_nf)
\\
&\le
C(q)^{1-\alpha}\lim_{n\to\infty}V^{p, \alpha}_{\gamma_n}(\mathbb{E}_nf)
=C(q)^{1-\alpha}V^{p, \alpha}_\gamma(f).
\end{align*}
Hence
$
U^{p,\alpha}_\gamma(f)\le C(q)^{1-\alpha}V^{p, \alpha}_\gamma(f).
$

$\rm (ii)\Rightarrow (i)$.
Assume now that $U^{p, \alpha}_\gamma(f)<\infty$. Then, for any vector field $\Phi$ of the form
$$
\Phi(x) = \sum_{i=1}^n\Psi_i(l_1(x), \ldots, l_n(x))e_i,
$$
where $\Psi\in C_0^\infty(\mathbb{R}^n, \mathbb{R}^n)$, one has
$$
\int_{\mathbb{R}^n} \langle\Psi, \nabla T^n_t\mathbb{E}_n f\rangle\, d\gamma_n
=
\int_{\mathbb{R}^n} \langle\Psi, \mathbb{E}_n\nabla T_t f\rangle \, d\gamma_n
=
\int_X \langle\Phi, \nabla T_t f\rangle_H \, d\gamma
\le
\|\Phi\|_{L^q(\gamma)}\|\nabla T_t f\|_{L^p(\gamma)}.
$$
Thus, $\|\nabla T^n_t\mathbb{E}_n f\|_{L^p(\gamma_n)}\le\|\nabla T_t f\|_{L^p(\gamma)}$,
since $\|\Psi\|_{L^q(\gamma_n)}=\|\Phi\|_{L^q(\gamma)}$.
Hence $U^{p,\alpha}_{\gamma_n}(\mathbb{E}_nf)\le U^{p,\alpha}_\gamma(f)$ and
by Theorem \ref{T2.2} we have $V_{\gamma_n}^{p,\alpha}(\mathbb{E}_nf)<\infty$.
Moreover,
$$
V^{p, \alpha}_{\gamma_n}(\mathbb{E}_nf)\le(2C(p)+1)U^{p, \alpha}_{\gamma_n}(\mathbb{E}_nf)\le (2C(p)+1)U^{p,\alpha}_\gamma(f).
$$
Since $V^{p, \alpha}_{\gamma_n}(\mathbb{E}_nf)\to V^{p, \alpha}_\gamma(f)$,
we obtain the claim.
\end{proof}

We now compare the Gaussian classes $B^\alpha_p(\gamma)$ with other known
scales of Gaussian fractional Sobolev classes (there are also connections with
the scales in \cite{Pineda}, \cite{Nikit1}, \cite{Nikit2},
which will be considered elsewhere).

For $f\in L^p(\gamma)$ and $\alpha>0$ let
$$
V_\alpha(f):= \Gamma(\alpha/2)^{-1}\int_0^\infty t^{\alpha/2 - 1}e^{-t}T_tf\, dt
 = (I - L)^{-\alpha/2}f.
 $$
We recall (see, e.g., \cite{GM}, \cite{DM}, and~\cite{Shig})
that the scale of Sobolev classes $H^{p,\alpha}(\gamma)$ is defined by
$$
H^{p,\alpha}(\gamma):= V_\alpha(L^p(\gamma)),\quad \|f\|_{H^{p,\alpha}(\gamma)}:= \|V_\alpha^{-1}f\|_p.
$$
Now for $r<0$ we define  $H^{p,r}(\gamma)$ to be the dual to
$H^{p',-r}(\gamma)$, $p'=p/(p-1)$.

For $f\in L^p(\gamma)$ we also set
$$
K_t(f) = \inf\{\|f_1\|_p + t\|f_2\|_{W^{p,1}(\gamma)} \colon \ f=f_1 + f_2, \, f_1\in L^p(\gamma),\, f_2\in W^{p,1}(\gamma)\}
$$
and, for $\alpha\in (0,1)$,
consider the class $\mathcal{E}^{p,\alpha}(\gamma)$ of functions with finite norm
$$
\|f\|_{\mathcal{E}^{p,\alpha}(\gamma)}:=
\biggl(\int_0^\infty |t^{-\alpha} K_t(f)|^pt^{-1}\, dt\biggr)^{1/p}.
$$
Similarly, by the same interpolation method one defines the classes
$\mathcal{E}^{p,\alpha}(\gamma)$  for all real~$\alpha$.

As shown by Watanabe \cite{Watanabe}, for all $p>1$, $\alpha\in\mathbb{R}$, and $\varepsilon>0$,
there hold continuous embeddings
$$
H^{p,\alpha+\varepsilon}(\gamma)\subset  \mathcal{E}^{p,\alpha}(\gamma) \subset H^{p,\alpha-\varepsilon}(\gamma).
$$

\begin{theorem}
Let $\alpha\in (0,1)$,  $p\in(1, \infty)$. For any $\beta<\alpha$ we have
$$
H^{p,\alpha}(\gamma)\subset B^\alpha_p(\gamma)\subset \mathcal{E}^{p, \beta}(\gamma).
$$
Moreover, there are number $C_1= C_1(p, \alpha,\beta)$ and
$C_2= C_2(p, \alpha)$
depending only on the indicated parameters such that
$$
\|f\|_{\mathcal{E}^{p,\beta}(\gamma)}
\le
C_1\|f\|_p\max\{1, [V^{p, \alpha}_\gamma(f)]^{\beta/\alpha}\|f\|_p^{-\beta/\alpha}\},
$$
$$
V^{p, \alpha}_\gamma(f)
\le
C_2\|f\|_{H^{p,\alpha}(\gamma)}.
$$
Therefore, for all $\varepsilon>0$ we have continuous embeddings
$$
H^{p,\alpha}(\gamma)\subset B^\alpha_p(\gamma)\subset H^{p,\alpha-\varepsilon}(\gamma).
$$
\end{theorem}
\begin{proof}
We observe that
\begin{align*}
K_t(f)
&\le \|f-T_{t^2}f\|_p + t\|T_{t^2}f\|_{W^{p,1}(\gamma)}\le 2^{1-\alpha}C(p)^\alpha t^\alpha\ V_\gamma^{p,\alpha}(f)
+ t\|T_{t^2}f\|_p + t\|\nabla T_{t^2}f\|_p
\\
&\le
t\|f\|_p +
[2^{1-\alpha}C(p)^\alpha + C(q)^{1-\alpha}]V^{p, \alpha}_\gamma(f)t^\alpha.
\end{align*}
Set $C_0(p, \alpha):=2^{1-\alpha}C(p)^\alpha + C(q)^{1-\alpha}$.
For any $R>0$ we have
\begin{align*}
&\|f\|_{\mathcal{E}^{p,\beta}(\gamma)}= \biggl(\int_0^\infty |t^{-\beta} K_t(f)|^pt^{-1}\, dt\biggr)^{1/p}=
\biggl(\int_0^R t^{-p\beta-1} K_t(f)^p\, dt + \int_R^\infty t^{-p\beta-1} K_t(f)^p\, dt\biggr)^{1/p}
\\
&\le
\biggl(\int_0^R t^{-p\beta-1} (t\|f\|_p + C_0(p, \alpha)V^{p, \alpha}_\gamma(f)t^\alpha)^p\, dt\biggr)^{1/p}+
\|f\|_p\biggl(\int_R^\infty t^{-p\beta-1}\,dt\biggr)^{1/p}
\\
&\le
\|f\|_p\biggl(\int_0^R t^{p(1-\beta)-1}\,dt\biggr)^{1/p}
+
C_0(p, \alpha)V^{p, \alpha}_\gamma(f)\biggl(\int_0^R t^{p(\alpha-\beta)-1}\,dt\biggr)^{1/p}
+
\|f\|_p\biggl(\int_R^\infty t^{-p\beta-1}\,dt\biggr)^{1/p}
\\
&=
(p(1-\beta))^{-1/p}\|f\|_pR^{1-\beta}
+
C_0(p, \alpha)(p(\alpha - \beta))^{-1/p}V^{p, \alpha}_\gamma(f)R^{\alpha-\beta}
+
(p\beta)^{-1/p}\|f\|_pR^{-\beta}.
\end{align*}
If $V^{p, \alpha}_\gamma(f)\le \|f\|_p$, then taking $R=1$ we obtain
$$
\|f\|_{\mathcal{E}^{p,\beta}(\gamma)} \le C_1(p, \alpha,\beta)\|f\|_p
$$
with $C_1(p, \alpha,\beta) = (p(1-\beta))^{-1/p} + C_0(p, \alpha)(p(\alpha - \beta))^{-1/p} + (p\beta)^{-1/p}$.
If $V^{p, \alpha}_\gamma(f)\ge \|f\|_p$ we can take
$R = (\|f\|_p [V^{p, \alpha}_\gamma(f)]^{-1})^{1/\alpha}\le1$ and conclude that
$$
\|f\|_{\mathcal{E}^{p,\beta}(\gamma)}\le
C_1(p, \alpha,\beta)[V^{p, \alpha}_\gamma(f)]^{\beta/\alpha}\|f\|_p^{1-\beta/\alpha}.
$$
Therefore,
$$
\|f\|_{\mathcal{E}^{p,\beta}(\gamma)}
\le
C_1(p, \alpha,\beta)\|f\|_p\max\{1, [V^{p, \alpha}_\gamma(f)]^{\beta/\alpha}\|f\|_p^{-\beta/\alpha}\},
$$
as announced.

Let $f = V_\alpha (g)$, $g\in L^p(\gamma)$. Then for any $\Phi\in\mathcal{FC}^\infty$ and
$R>0$ we have
\begin{align*}
&\int_X {\rm div}_\gamma\Phi f \, d\gamma
=
\int_X {\rm div}_\gamma\Phi V_\alpha (g)\, d\gamma
\\
&=
\Gamma(\alpha/2)^{-1}\int_0^Rt^{\alpha/2 - 1}e^{-t}
\int_X {\rm div}_\gamma\Phi T_tg\, d\gamma\, dt
-
\Gamma(\alpha/2)^{-1}\int_R^\infty t^{\alpha/2 - 1}e^{-t}
\int_X \langle\Phi, \nabla T_tg\rangle \, d\gamma\, dt
\\
&\le
\Gamma(\alpha/2)^{-1}\|g\|_p\|{\rm div}_\gamma\Phi\|_q\int_0^Rt^{\alpha/2 - 1}e^{-t}\,dt
+
\Gamma(\alpha/2)^{-1}C(q)\|g\|_p\|\Phi\|_q\int_R^\infty t^{\alpha/2 - 1}e^{-t}\frac{e^{-t}}{\sqrt{1-e^{-2t}}}\ dt
\\
&\le
\Gamma(\alpha/2)^{-1}(2/\alpha)\|g\|_p\|{\rm div}_\gamma\Phi\|_qR^{\alpha/2}
+
\Gamma(\alpha/2)^{-1}C(q)\|g\|_p\|\Phi\|_q\int_R^\infty t^{\alpha/2 - 1-1/2}e^{-t}\, dt
\\
&\le
\Gamma(\alpha/2)^{-1}(2/\alpha)\|g\|_p\|{\rm div}_\gamma\Phi\|_qR^{\alpha/2}
+
\Gamma(\alpha/2)^{-1}2(1-\alpha)^{-1}C(q)\|g\|_p\|\Phi\|_qR^{\alpha/2-1/2}.
\end{align*}
We now take $R = (\|\Phi\|_q\|{\rm div}_\gamma\Phi\|_q^{-1})^2$ and obtain
$$
\int_X {\rm div}_\gamma\Phi f\, d\gamma
\le
C(p, \alpha)\|g\|_p \|{\rm div}_\gamma\Phi\|_q^{1-\alpha}\|\Phi\|_q^\alpha,
$$
where
$C(p, \alpha) = \Gamma(\alpha/2)^{-1}(2/\alpha) + 2\Gamma(\alpha/2)^{-1}(1-\alpha)^{-1}C(q)$.
Hence we have $f\in B^\alpha_p(\gamma)$ and $V^{p, \alpha}_\gamma(f) \le C(p, \alpha)\|g\|_p$.
\end{proof}

\begin{remark}
{\rm
We draw the reader's attention to the fact that in the case of a Gaussian measure
in place of Lebesgue measure we introduce two  definitions with certain similarities
to the Lebesgue case (integration by parts and  semigroup estimates), but do not consider
the straightforward analog of the $L^p$-bound of the difference $f_h-f$.
Of course, the point is that Gaussian measures are not translation invariant.
As a result, shifts of a function in $L^p$ may fail to belong to~$L^p$. However,
in principle it is possible to study functions in $L^p(\gamma)$ with the property that
$\|f_h-f\|_{L^p(\gamma)}\le C|h|^\alpha$. For example, if $\alpha=1$ and $p>1$,
then the estimate
$\|f_h-f\|_{L^p(\gamma)}\le C|h|$
yields that $f$ has a Sobolev derivative $\partial_h f\in L^p(\gamma)$, which can be obtained
as the limit of a subsequence in $n(f_{h/n}-f)$ weakly converging in~$L^p(\gamma)$
(recall that any sequence bounded in $L^p(\gamma)$ contains a weakly converging subsequence).
In the finite-dimensional case this shows that $f\in W^{p,1}(\gamma)$, but in infinite dimensions
$f$ belongs to some larger weak Sobolev class.

Conversely, if $f\in W^{p,1}(\gamma)$, then, for any $s<p$ and $h\in H$ with $|h|\le 1$,
it is easy to obtain the bound
$$
\|f_h-f\|_{L^s(\gamma)}\le C |h| \| \nabla f\|_{L^p(\gamma,H)}
$$
with an absolute constant $C$.
This can be done by expressing $f_h-f$ through the integral of $\nabla f(x-th)$
and estimating $\|\nabla f(\cdot -th)\|_{L^s(\gamma)}$ via $\|\nabla f\|_{L^p(\gamma)}$
by means of the Cameron--Martin formula for the density of the shifted measure and H\"older's inequality
(actually, this estimate involves $\|\partial_h f\|_{L^p(\gamma)}$, hence remains valid for the
aforementioned larger weak Sobolev classes).

It is also natural to consider
bounds on the total variation of the measure $(f\cdot\gamma)_h-f\cdot\gamma$,
which is in the spirit of the next section and (in case $\alpha=1$) of constructions
in \cite{BReb1} and \cite{BReb2}.
}\end{remark}

\section{The Nikolskii--Besov smoothness of measures}\label{sect7}

In this section we introduce the Nikolskii--Besov $\alpha$-smoothness of measures on locally convex spaces.
This is also of interest in the finite-dimensional case because enables us to pass from functions
to measures when considering the directional smoothness.

For a Radon measure $\mu$ on a locally convex space $X$ and a vector $h\in X$ the measure $\mu_h$
is defined by
$$
\mu_h(A) := \mu (A-h).
$$

\begin{definition}\label{D4.1}
Let $\alpha\in (0,1]$.
A Radon measure $\mu$ on a locally convex space $X$ is called $\alpha$-H\"older
along a vector $h\in E$ if
$$
\|\mu_{th} - \mu\|_{\rm TV} \le C|t|^\alpha \quad \forall\, t\in\mathbb{R}
$$
for some number $C$.
The infimum of such numbers $C$ is denoted by $V^\alpha(\mu; h)$.
\end{definition}

For $\alpha = 1$ the previous definition coincides with the definition of the Skorohod differentiability
of  $\mu$ along $h$ (see \cite[Section 3.1]{DM}). It is known that the latter is equivalent to the existence
of the Skorohod derivative $d_h\mu$ of $\mu$ defined as the limit of the measures $t^{-1}(\mu_{th}-\mu)$
in the weak topology as $t\to 0$. It is also equivalent
to the estimate
$$
\biggl|\int_X \partial_h f(x)\, \mu(dx)\biggr|\le C \sup_x |f(x)| \quad \forall\, f\in \mathcal{FC}(X).
$$

\begin{definition}\label{D4.2}
Let $D^\alpha(\mu)$ be the set of all vectors $h\in X$ such that
$\mu$ is $\alpha$-H\"older along~$h$.
\end{definition}

In case $\alpha=1$ we obtain the so-called subspace $D_C(\mu)$ of Skorohod differentiability.
For a nonzero measure $\mu$ the subspace of Skorohod differentiability can be equipped with a norm
with respect to which it becomes a Banach space whose closed unit ball is compact in~$X$
 (see \cite[Theorem 5.1.1]{DM}).
An analog of this is proved below for $D^\alpha(\mu)$.

The following proposition is a corollary of the Nikodym theorem (see \cite[Corollary~4.6.4]{mera}).

\begin{proposition}
A Radon measure $\mu$ is $\alpha$-H\"older along $h\in X$ if and only if
for every Borel set $A$ there is a number $C(A)$ such that
$$
|\mu(A+th) - \mu(A)| \le C(A)|t|^\alpha \quad \forall\, t.
$$
This is also equivalent to the
property that the function $t\mapsto\mu(A+th)$ is H\"older of order $\alpha$.
\end{proposition}

The following theorem describes the structure of the space $D^\alpha(\mu)$
in the spirit of the aforementioned
theorem for the space of Skorohod differentiability.
It is worth noting that according to another known result (see \cite[Chapter~5]{DM}),
the space $C(\mu)$ of all vectors of continuity of a nonzero Radon measure $\mu$
(vectors $h$ such that $\lim\limits_{t\to 0} \|\mu_{th}-\mu\|=0$) is a linear subspace
in~$X$ that is complete with respect to the metric
$d_0(a,b)=\sup_{|t|\le 1} \|\mu_{ta}-\mu_{tb}\|$ and closed balls of a sufficiently small radius
are compact in~$X$.

\begin{theorem}\label{T4.1}
Let $\mu$ be a nonzero Radon measure.
The space $D^\alpha(\mu)$ equipped
with the translation invariant metric
$$
d_{D^\alpha(\mu)}(h_1,h_2):=V^\alpha(\mu; h_1-h_2)
$$
is a complete metric vector space compactly embedded into the space~$X$.
\end{theorem}
\begin{proof}
It is easily verified that
$D^\alpha(\mu)$ is a linear subspace and
$V^\alpha(\mu; h)$ is indeed a metric.
Clearly, $D^\alpha(\mu)\subset C(\mu)$ and the identity embedding of
$(D^\alpha(\mu), d_{D^\alpha(\mu)})$ into $(C(\mu),d_0)$ is continuous.
Hence the embedding into $X$ is compact.

Let now $\{h_n\}$ be a Cauchy sequence in $(D^\alpha(\mu), d_{D^\alpha(\mu)})$.
This sequence is bounded in $D^\alpha(\mu)$. By the compactness of embedding into $X$
there is a subsequence $\{h_{n_k}\}$ which converges in $X$ to some vector~$h$.
For every function $\varphi\in C_b(X)$ with $\sup|\varphi|\le1$ we have
$$
\int_X (\varphi(x+th_n) - \varphi(x + th_{n_k}))\, \mu(dx)\le \|\mu_{th_n} - \mu_{th_{n_k}}\|_{\rm TV}.
$$
By Lebesgue's dominated convergence theorem the left-hand side tends to
$$
\int_X (\varphi(x+th_n) - \varphi(x+th))\, \mu(dx),
$$
which yields that
$$
\|\mu_{th_n} - \mu_{th}\|_{\rm TV}\le\liminf_{k\to\infty}\|\mu_{th_n} - \mu_{th_{n_k}}\|_{\rm TV}.
$$
Since the sequence $\{h_n\}$ is Cauchy in $D^\alpha(\mu)$,
 for every fixed $\varepsilon>0$, for all $n$ large enough we have
$$
\|\mu_{th_n} - \mu_{th}\|_{\rm TV}\le\varepsilon |t|^\alpha.
$$
Therefore, $h_n - h\in D^\alpha(\mu)$ and $h_n\to h$ in $D^\alpha(\mu)$. Thus, $D^\alpha(\mu)$ is
complete. A similar reasoning shows that closed balls in $D^\alpha(\mu)$ are closed in~$X$, hence
they are compact in~$X$.
\end{proof}

We recall that for every Radon measure $\mu$ on $X$ and every nonzero vector $h$, we can
write $X$ as a direct topological sum $X=\mathbb{R}h\oplus Y$ of the one-dimensional space
generated by $h$ and a closed hyperplane $Y$ such that there exist measures $\mu^y$, $y\in Y$,
on the straight lines $\mathbb{R}h+y$ representing $\mu$ in the form
$$
\mu(B)=\int_Y \mu^y(B)\, |\mu|_Y(dy)
$$
for every Borel set $B$, where $|\mu|_Y$ is the projection of $|\mu|$ on~$Y$
(and the functions $y\mapsto \mu^y(B)$ are Borel measurable). The measures $\mu^y$ are called conditional measures.
They are defined uniquely up to redefining on a set of $|\mu|_Y$-measure zero.

It is sometimes convenient to define the conditional measures $\mu^y$ on the same line $\mathbb{R}h$
 in place
of different parallel straight lines. This can be easily done by passing to the measures
$\nu^y(A):=\mu^y(A+y)$ for $A\subset \mathbb{R}h$; in that case
  $\mathbb{R}h$ can be identified with~$\mathbb{R}$.

Due to Example \ref{ex1} we know that even on the plane conditional measures for an $\alpha$-H\"older measure
need not  be $\alpha$-H\"older of the same order $\alpha$.
Nevertheless, the following proposition shows that they are $\beta$-H\"older
with an arbitrary order $\beta<\alpha$.

\begin{proposition}
Suppose that  a Radon measure $\mu$ is $\alpha$-H\"older along a vector $h$.
Then $|\mu|_Y$-almost every
conditional measure $\mu^y$ is  $\beta$-H\"older for every $\beta<\alpha$.
\end{proposition}
\begin{proof}
The shifts $\mu^y_{th}$ of conditional measures will be denoted by $\mu^y_t$.
Then $\mu_{th}$ is represented as $\mu_t^y\, |\mu|_Y(dy)$. Hence
we have (see, e.g., \cite[p.~21]{DM})
$$
\|\mu_{th} - \mu\|_{\rm TV} = \int_Y \|\mu^y_t - \mu^y\|_{\rm TV}\, |\mu|_Y(dy).
$$
Hence, for every $\beta<\alpha$ and $t_n = 2^{-n}$, the series
$$
\sum_{n=1}^\infty\int_Y t_n^{-\beta}\|\mu^y_{t_n} - \mu^y\|_{\rm TV}\, |\mu|_Y(dy)
$$
converges. By the monotone convergence theorem we have $|\mu|_Y$-almost everywhere
$$
\sum_{n=1}^\infty t_n^{-\beta}\|\mu^y_{t_n} - \mu^y\|_{\rm TV}<\infty.
$$
Therefore, for $|\mu|_Y$-almost every $y$, there is a number $C(y)$ such that
$$
\|\mu^y_{t_n} - \mu^y\|_{\rm TV}\le C(y)t_n^\beta
$$
for every $n\in\mathbb{N}$. For any number $s\in[0, 1)$
take numbers $\varepsilon_n\in\{0, 1\}$ such that $s = \sum_{n=1}^\infty \varepsilon_n2^{-n}$.
Assume that $\varepsilon_{n_0}$ is the first nonzero number in this expansion.
Then, setting
$$
s_n = \sum_{k=1}^n\varepsilon_n2^{-n}
\quad \hbox{and} \quad  \mu^y_{s_0}=\mu^y,
$$
for this number $s$ we have
\begin{align*}
\|\mu^y_s - \mu^y\|_{\rm TV} &= \Bigl\|\sum_{n=n_0}^\infty\mu^y_{s_n} - \mu^y_{s_{n-1}}\Bigr\|_{\rm TV}
\\
&
\le
\sum_{n=n_0}^\infty\|\mu^y_{s_n} - \mu^y_{s_{n-1}}\|_{\rm TV}
=\sum_{n=n_0}^\infty\|\mu^y_{\varepsilon_n 2^{-n}} - \mu^y\|_{\rm TV}
\le
C(y)\sum_{n=n_0}^\infty 2^{-\beta n}.
\end{align*}
Let us observe that
$$
\sum_{n=n_0}^\infty 2^{-\beta n} =
2^{-n_0\beta}(1 - 2^{-\beta})^{-1}\le (1 - 2^{-\beta})^{-1}s^{\beta}.
$$
Therefore, for any number $s\in \mathbb{R}$ we have
$$
\|\mu^y_s - \mu^y\|_{\rm TV}\le \max\{2, (1 - 2^{-\beta})^{-1}C(y)\} |s|^\beta,
$$
which is the required property for the measure $\mu^y$.
\end{proof}

\end{document}